\documentclass[reqno]{amsart}
\usepackage[T1]{fontenc}
\usepackage{amssymb}
\usepackage[left=1in,right=1in,top=1in,bottom=1in]{geometry}
\usepackage[ddmmyyyy]{datetime}
\usepackage{enumitem}
\usepackage{physics}
\usepackage[foot]{amsaddr}
\usepackage{bm}
\usepackage{tikz}
\usepackage{mathrsfs}
\usepackage{mathtools}
\usepackage[hidelinks]{hyperref}
\usepackage{scalerel,stackengine}
\stackMath
\usetikzlibrary{arrows.meta}
\usepackage{subcaption}
\usepackage{float}
\usepackage[style=numeric,maxbibnames=99]{biblatex}
\usepackage{comment}

\newenvironment{mycases}
   {\begin{dcases*}}
   {\end{dcases*}}
\theoremstyle{plain}
\newtheorem{theorem}{Theorem}[section]
\newtheorem{lemma}[theorem]{Lemma}

\newtheorem{proposition}[theorem]{Proposition}

\newtheorem*{corollary*}{Corollary}
\theoremstyle{remark}
\newtheorem{remark}[theorem]{Remark}
\counterwithout{figure}{section}
\theoremstyle{definition}
\newtheorem{definition}[theorem]{Definition}

\newcommand{\br}[1]{\left(#1\right)}

\newcommand{\ak}{\abs{k}}

\newcommand{\aal}{\abs{\alpha}}
\newcommand{\aab}{\abs{\beta}}
\newcommand{\sbr}[1]{\left[#1\right]}

\newcommand{\p}{\partial}
\newcommand{\ep}{\varepsilon}
\newcommand{\jb}[1]{\left\langle#1\right\rangle}
\newcommand{\set}[1]{\left\{#1\right\}}
\newcommand{\mytextfrac}[2]{\textstyle\frac{#1}{#2}}
\newcommand{\qaq}{\quad\text{and}\quad}

\newcommand{\R}{\mathbb{R}}
\newcommand{\C}{\mathbb{C}}

\newcommand{\lam}{\lambda}

\newcommand{\fin}{f_{\mathrm{in}}}
\newcommand{\Qin}{Q_{\mathrm{in}}}
\newcommand{\hin}{h_{\mathrm{in}}}
\newcommand\reallywidehat[1]{%
\savestack{\tmpbox}{\stretchto{%
  \scaleto{%
    \scalerel*[\widthof{\ensuremath{#1}}]{\kern.pt\mathchar"0362\kern.1pt}%
    {\rule{0ex}{\textheight}}
  }{\textheight}%
}{2.4ex}}%
\stackon[-6.9pt]{#1}{\tmpbox}%
}

\numberwithin{equation}{section}

\title[The semiclassical limit from Hartree to Vlasov at positive density]{\small{\mbox{The semiclassical limit from Hartree to Vlasov at positive density:} strong uniform-in-time convergence and scattering}}
\author{Marnie Smith}
\address{Department of Pure Mathematics and Mathematical Statistics, University of Cambridge}
\email{ms2724@cam.ac.uk}
\begin{document}
\begin{abstract}Strong semiclassical convergence from the Hartree equation to the Vlasov equation is established in three dimensions near Penrose-stable homogeneous steady states at positive density. For sufficiently regular integrable interaction kernels, an $O(\hbar^2)$ convergence rate in weighted Sobolev spaces is proved on finite time intervals. Combining this estimate with earlier uniform-in-$\hbar$ phase-mixing and scattering bounds for the Hartree equation yields global existence and scattering for the corresponding Vlasov solution through the semiclassical limit. If the quantum and classical initial perturbations are aligned, the semiclassical convergence is uniform for all times, and the Wigner transforms of the quantum scattering profiles converge to the classical scattering profile at the explicit rate $O((\ln\hbar^{-1})^{-1/3})$.
\end{abstract}
\date{5 September 2026}
\maketitle
\tableofcontents

\section{Introduction}
\subsection{The Hartree and Vlasov equations}
The \textit{Hartree equation} describes the quantum mean-field evolution of an interacting many-particle system. The state of the system is represented by a density matrix $\gamma^\hbar(t)$, which is a self-adjoint operator on $L^2(\R^3;\C)$, and evolves by
\begin{align}\label{H}
    \begin{mycases}
        i \hbar\p_t\gamma^{\hbar} = \sbr{-\frac{\hbar^2}{2}\Delta + K*_x\rho_{\gamma^{\hbar}}, \gamma^{\hbar}}, \\
        \gamma^{\hbar}(0) = \gamma^{\hbar}_{\mathrm{in}},
    \end{mycases}
\end{align}
where $\hbar>0$ is the reduced Planck constant, $K\colon\R^3\to\R$ is a pairwise interaction kernel, $\gamma^\hbar(t,x,y)$ denotes the integral kernel of $\gamma^\hbar(t)$, and $\rho_{\gamma^{\hbar}}(t,x):=\gamma^{\hbar}(t,x,x)$ is the associated spatial density. In the fermionic setting, equation~\eqref{H} is also known as the reduced Hartree--Fock equation, since the exchange term is omitted. In the positive-density setting considered here, the density matrices have infinite trace and represent an infinite gas of particles~\cite{LewinSabinI2015}.

The \textit{Vlasov equation} is the corresponding classical kinetic model for a system of collisionless particles, whose state is a distribution function $f(t,x,\xi)$ on phase space, evolving by
\begin{align}\label{V}
    \begin{mycases}
    \p_tf+\xi\cdot\nabla_xf-(\nabla K*_x\rho_f)\cdot\nabla_\xi f=0,\\
    f(0,x,\xi)=\fin(x,\xi),
    \end{mycases}
\end{align}
where $x\in\R^3$ and $\xi\in\R^3$ denote position and velocity, and $\rho_f(t,x):=\int_{\R^3}f(t,x,\xi)\,d\xi$ is the associated spatial density. The two models are formally related by the semiclassical limit $\hbar\to0$; the correspondence is expressed by the Wigner transform, introduced in the next subsection.

Both equations admit translation-invariant equilibria: any $\mu\in L^1(\R^3;\R_+)$ defines steady states of~\eqref{H} and~\eqref{V}, given respectively by
\begin{align*}
\gamma^\hbar_\mu:=(2\pi\hbar)^3\mu(-i\hbar\nabla) \qaq f_\mu(x,\xi):=\mu(\xi).
\end{align*}
The operator $\gamma^\hbar_\mu$ has integral kernel
\begin{align}\label{intkernelofsteadystate}
\gamma^\hbar_\mu(x,y)=\int_{\R^3}e^{i\xi\cdot(x-y)/\hbar}\mu(\xi)\,d\xi,
\end{align}
and hence, evaluating on the diagonal, its spatial density is the constant $\int_{\R^3}\mu(\xi)\,d\xi$; the distribution function $f_\mu$ has the same density. The mean-field potential therefore reduces to the constant $\br{\int_{\R^3}K}\br{\int_{\R^3}\mu}$, and both states are steady: $\gamma^\hbar_\mu$ commutes with the free Hamiltonian, and $f_\mu$ is independent of $x$. If $\mu\not\equiv0$, the density is strictly positive and both states have infinite mass: $\Tr\gamma^\hbar_\mu=+\infty$ and $f_\mu\notin L^1(\R^3\times\R^3)$.

Writing $\gamma^\hbar=\gamma^\hbar_\mu+Q^\hbar$ in~\eqref{H} yields the nonlinear evolution of the quantum perturbation,
\begin{align}\label{NH}
    \begin{mycases}
        i\hbar\p_tQ^\hbar=\sbr{-\tfrac{\hbar^2}{2}\Delta,Q^\hbar}+\sbr{K*_x\rho_{Q^\hbar},Q^\hbar+\gamma^\hbar_\mu},\\
        Q^\hbar(0)=\Qin^\hbar:=\gamma^\hbar_{\mathrm{in}}-\gamma^\hbar_\mu,
    \end{mycases}
\end{align}
while writing $f=f_\mu+h$ in~\eqref{V} yields that of the classical perturbation,
\begin{align}\label{NV}
\begin{mycases}
    \p_th+\xi\cdot\nabla_x h-(\nabla K*_x\rho_h)\cdot\nabla_\xi\br{h+\mu}=0,\\
    h(0)=\hin:=\fin-f_\mu.
\end{mycases}
\end{align}
This paper studies the semiclassical limit from~\eqref{NH} to~\eqref{NV} near the matched homogeneous steady states. Strong quantitative convergence at rate $O(\hbar^2)$ is first established in weighted Sobolev spaces on any finite time interval on which the corresponding solutions satisfy the required regularity bounds. This result is then combined with uniform-in-$\hbar$ phase-mixing and scattering estimates for the Hartree equation obtained in earlier work~\cite{smith2025phasemixinghartreeequation}. In the Penrose-stable regime, these quantum estimates are transferred to the Vlasov equation, yielding convergence uniformly for all times, including at the level of the scattering profiles.

\subsection{Phase-space formulation and free dynamics}
The semiclassical comparison uses the Wigner transform to represent quantum operators as functions on phase space and to measure the quantum and classical perturbations in matched Sobolev norms. The same transform carries the free Schrödinger evolution to free transport, allowing both equations to be studied in a common moving frame.

Throughout, use the Japanese bracket notation $\jb{\eta} := (1+\abs{\eta}^2)^{1/2}$ and $\jb{k,\eta}:=(1+\ak^2+\abs{\eta}^2)^{1/2}$; elementary inequalities for these brackets are collected in Lemma~\ref{japanesebracketlemma} of Appendix~\ref{appendixB}. For $\mu(\xi)$ and $f(x,\xi)$, adopt the Fourier convention
\begin{align}
    \widehat \mu(\eta):=\int_{\R^3} e^{-i\eta\cdot \xi}\mu(\xi)\,d\xi,\qquad\widehat f(k,\eta):=\int_{\R^3}\int_{\R^3} e^{-ik\cdot x}e^{-i\eta\cdot \xi}f(x,\xi)\,d\xi dx.\label{FTconvention}
\end{align}

The quantum and classical perturbations are measured in Sobolev spaces that control
both phase-space regularity and polynomial moments in the velocity variable; here
$\jb{\nabla_x,\nabla_\xi}^\sigma$ and $\jb{\nabla_\xi}^\sigma$ denote the Fourier
multipliers $\jb{k,\eta}^\sigma$ and $\jb{\eta}^\sigma$.
\begin{definition}[Sobolev norms]\label{Sobolevnormsdefinition}
For $\sigma\ge0$, an integer $M\ge0$ and functions $f(x,\xi)$ and $\mu(\xi)$, define
\[
\norm{f}_{H^\sigma_M}^2:=\sum_{|\alpha|\le M}\norm{\jb{\nabla_x,\nabla_\xi}^\sigma(\xi^\alpha f)}_{L^2}^2
\qaq
\norm{\mu}_{H^\sigma_M}^2:=\sum_{|\alpha|\le M}\norm{\jb{\nabla_\xi}^\sigma(\xi^\alpha\mu)}_{L^2}^2.
\]
\end{definition}

Density matrices are compared with phase-space functions through the Wigner transform.
\begin{definition}[Wigner transform]\label{Wignerdefinition}
For an operator $\gamma$ with kernel $\gamma(x,y)$, set
\[
W^{\hbar}[\gamma](x,\xi)=(2\pi\hbar)^{-3}\int_{\R^3} e^{-i\xi\cdot y/\hbar}\gamma(x+\tfrac{y}{2},x-\tfrac{y}{2})\,dy.
\]
\end{definition}

The quantum Sobolev norm of an operator is defined as the classical Sobolev norm of its Wigner transform. By construction, the quantum--classical comparison is then isometric at the level of the norms.
\begin{definition}[Quantum Sobolev norms]\label{quantumnormsdefinition}
For $\sigma\geq0$, an integer $M\geq0$ and an operator $\gamma$ on $L^2(\R^3)$, define
\begin{align*}
    \norm{\gamma}_{\mathcal L^2}:=\norm{W^{\hbar}[\gamma]}_{L^2}
    \qaq
    \norm{\gamma}_{\mathcal H^\sigma_M}:=\norm{W^{\hbar}[\gamma]}_{H^\sigma_M};
\end{align*}
in particular, $\gamma\in\mathcal H^\sigma_M$ if and only if $W^{\hbar}[\gamma]\in H^\sigma_M$, and $\mathcal L^2=\mathcal H^0_0$. 
\end{definition}

\begin{remark}[Hilbert--Schmidt norm]\label{HSnormremark}
The norm $\mathcal L^2$ is the semiclassically rescaled Hilbert--Schmidt norm:
\begin{align*}
    \norm{\gamma}_{\mathcal L^2}=(2\pi\hbar)^{-3/2}\norm{\gamma}_{\mathrm{HS}}.
\end{align*}
Indeed, writing $W^{\hbar}[\gamma](x,\xi)=(2\pi\hbar)^{-3}\mathcal F\sbr{\gamma\br{x+\mytextfrac{\cdot}{2},x-\mytextfrac{\cdot}{2}}}\br{\xi/\hbar}$, Plancherel's theorem gives $\norm{W^{\hbar}[\gamma]}_{L^2}=(2\pi\hbar)^{-3/2}\norm{\gamma(x,y)}_{L^2_{x,y}}$, and the $L^2_{x,y}$ norm of the integral kernel is precisely $\norm{\gamma}_{\mathrm{HS}}$.
\end{remark}

At integer order, the phase-space derivatives and weights also admit direct operator analogues. These are used when manipulating the Hartree equation at the operator level, before passing to Wigner transforms.
\begin{definition}[Quantum operators]\label{quantumoperators}
    The quantum analogues of spatial and velocity derivatives are
    \begin{align*}
        \bm{\nabla}_x\gamma :=[\nabla,\gamma ]&\qaq
        \bm{\nabla}_\xi \gamma :=[\mytextfrac{x}{i\hbar},\gamma ].
    \intertext{The quantum analogues of spatial and velocity weights are}
        \bm{x}\gamma:=\tfrac{1}{2}(x\gamma+\gamma x)&\qaq\bm{\xi}\gamma :=-\mytextfrac{i\hbar}{2}(\nabla \gamma +\gamma \nabla).
    \end{align*} 
    These operators are used only at integer order. Fractional Sobolev regularity for operators is defined through the Wigner transform in Definition~\ref{quantumnormsdefinition}.
\end{definition}

The moving frame for the comparison is built from the two free evolutions.
\begin{definition}[Free dynamics]\label{freeequationsdefinition}
The free transport equation is
\begin{align*}
    \begin{mycases}
        \p_tf+\xi\cdot\nabla_xf=0,\\
        f(0,x,\xi)=\fin(x,\xi),
    \end{mycases}
\end{align*}
with solution $f(t,x,\xi)=\fin(x-t\xi,\xi)$.
The free Schrödinger equation is
\begin{align*}
    \begin{mycases}
        i\hbar\p_t\gamma^\hbar=\sbr{-\frac{\hbar^2}{2}\Delta,\gamma^\hbar},\\
        \gamma^\hbar(0)=\gamma^\hbar_{\rm in},
    \end{mycases}
\end{align*}
with solution $\gamma^\hbar(t)=e^{i\frac{\hbar}{2}t\Delta}\gamma^\hbar_{\rm in}e^{-i\frac{\hbar}{2}t\Delta}$, where $(e^{i\frac{\hbar}{2}t\Delta})_{t\in\R}$ denotes the free Schrödinger group on $L^2(\R^3;\C)$; see, for instance,~\cite[Chapter~2]{Cazenave2003}.
\end{definition}

The following lemma records the compatibility properties of the Wigner transform used throughout the paper: it matches the quantum and classical derivatives and weights, conjugates the free Schrödinger evolution to free transport, maps the homogeneous states onto one another, and identifies the spatial densities.
\begin{lemma}[Properties of the Wigner transform]\label{operatorsWigner}The Wigner transform of Definition~\ref{Wignerdefinition} has the following properties:
\begin{enumerate}[label=(\roman*)]
\item (Quantum derivatives and weights)\label{intertwiningproperty} For sufficiently regular $\gamma$,
\begin{align*}
    \nabla_xW^{\hbar}[\gamma]=W^{\hbar}[\bm{\nabla}_x\gamma],\quad\nabla_\xi W^{\hbar}[\gamma]=W^{\hbar}[\bm{\nabla}_\xi\gamma],\quad xW^{\hbar}[\gamma]=W^{\hbar}[\bm{x}\gamma],\quad\xi W^{\hbar}[\gamma]=W^{\hbar}[\bm{\xi}\gamma].
\end{align*}
Thus the integer-order quantum derivatives and weights agree with the corresponding classical operations under the Wigner transform.
    \item (Free evolution) For $\gamma\in\mathcal L^2$,\label{HartreetoFreeTransport}
    \begin{align*}
        W^{\hbar}\sbr{e^{i\frac{\hbar}{2}t\Delta}\gamma e^{-i\frac{\hbar}{2}t\Delta}}(x,\xi)&=W^{\hbar}[\gamma](x-t\xi,\xi),\\
        W^{\hbar}\sbr{e^{-i\frac{\hbar}{2}t\Delta}\gamma e^{i\frac{\hbar}{2}t\Delta}}(x,\xi)&=W^{\hbar}[\gamma](x+t\xi,\xi).
    \end{align*}
    \item (Homogeneous states)\label{HomoegeneousstatesandWignertransform} For $\gamma_\mu^\hbar=(2\pi\hbar)^3\mu(-i\hbar\nabla)$,
    \begin{align*}
        W^{\hbar}[\gamma_\mu^\hbar](x,\xi)=\mu(\xi).
    \end{align*}
    \item (Spatial density)\label{densityproperty} Suppose the kernel of $\gamma$ is continuous and $W^{\hbar}[\gamma](x,\cdot)\in L^1(\R^3)$ for every $x\in\R^3$. Then
    \begin{align*}
        \rho_{W^{\hbar}[\gamma]}(x)=\int_{\R^3}W^{\hbar}[\gamma](x,\xi)\,d\xi=\gamma(x,x)=\rho_\gamma(x).
    \end{align*}
\end{enumerate}
\end{lemma}
\begin{proof}
The identities in~\ref{intertwiningproperty} follow by direct computation from Definition~\ref{Wignerdefinition}; iterating them gives the corresponding integer-order identities.
The identities in~\ref{HartreetoFreeTransport} follow from the Fourier representation of the Wigner transform and are standard; see for instance \cite{LionsPaul1993}. Property~\ref{HomoegeneousstatesandWignertransform} follows by substituting the integral kernel~\eqref{intkernelofsteadystate} into Definition~\ref{Wignerdefinition} and applying Fourier inversion.
For~\ref{densityproperty}, fix $x$ and write $g_x(y):=\gamma(x+\tfrac{y}{2},x-\tfrac{y}{2})$, so that $W^{\hbar}[\gamma](x,\xi)=(2\pi\hbar)^{-3}\widehat{g_x}(\xi/\hbar)$. The substitution $\xi=\hbar p$ gives
\begin{align*}
    \int_{\R^3}W^{\hbar}[\gamma](x,\xi)\,d\xi=\frac{1}{(2\pi)^3}\int_{\R^3}\widehat{g_x}(p)\,dp=g_x(0)=\gamma(x,x),
\end{align*}
by Fourier inversion, justified since $\widehat{g_x}\in L^1(\R^3)$ (equivalently, $W^{\hbar}[\gamma](x,\cdot)\in L^1(\R^3)$) and $g_x$ is continuous.
\end{proof}

\subsection{Interaction kernel and stability}
Throughout, the interaction kernel is assumed to be real-valued and integrable, $K\in L^1(\R^3;\R)$. In particular, the present framework does not include non-integrable interactions such as the Coulomb kernel. The additional Fourier decay required in the analysis is
quantified by the seminorms
\begin{align*}
    [K]_1:=\sup_{k\in\R^3}\jb{k}\abs{\widehat K(k)},
    \qquad
    [K]_2:=\sup_{k\in\R^3}\jb{k}^2\abs{\widehat K(k)}.
\end{align*}
Clearly, $[K]_1\leq [K]_2$. In the estimates below, implicit constants arising from the multiplier bounds on $K$ depend only on the relevant finite seminorm specified in each statement.

A representative example satisfying these conditions is the screened Coulomb (or Yukawa) potential
\[
    K(x) = \frac{e^{-\alpha\abs{x}}}{\abs{x}}, \qquad \widehat{K}(k) = \frac{4\pi}{\abs{k}^2 + \alpha^2},
\]
for \( \alpha > 0 \). 

The kernel assumptions above are sufficient for the finite-time semiclassical comparison. The large-time result additionally requires linear stability of the homogeneous state, uniformly as $\hbar\to0$. This is encoded by the following uniform Penrose condition, previously used in the semiclassical Hartree setting in~\cite{smith2025phasemixinghartreeequation}.
\begin{definition}[Uniform Penrose condition] Let $K\colon\R^3\to\R$ and $\mu\in L^1(\R^3;\R_+)$. The pair $(K,\mu)$ satisfies the \textit{uniform Penrose condition} if there exist \(\kappa,\delta > 0\) such that
\begin{align}\label{Penrose2}
    \inf_{\hbar\in(0,\delta]}\inf_{\substack{\Re\lambda \geq 0 \\ k \in \R^3}} \abs{1 +  \frac{2}{\hbar} \widehat{K}(k) \int_0^\infty e^{-\lam t}\sin\br{\frac{1}{2}\hbar t \ak^2} \widehat{\mu}\br{kt}\,dt} \geq \kappa.
\end{align}
\end{definition}

\begin{remark}\label{remark:uniformpenroseimpliesclassical}
The uniform Penrose condition implies the classical one (in, for example, \cite{mou-vil-2011,BedrossianMasmoudiMouhot2018}) whenever $K\in L^1(\R^3)$ and $\mu\in H^{r}_2(\R^3)$ for some $r>2$, and is conversely implied by it, for $\hbar$ small and up to halving the stability constant, under slightly more decay; see Lemma~\ref{Penroseequivalence}.
\end{remark}

\subsection{Main results}
The first result, Theorem~\ref{maintheorem1}, establishes quantitative semiclassical convergence from the Hartree to the Vlasov equation near homogeneous steady states, with rate $O(\hbar^2)$ on finite time intervals. The second, Theorem~\ref{theoremonapplication}, combines the first with the scattering theory for the Hartree equation to prove, in the Penrose-stable regime, a convergence rate that is uniform in time and, in particular, semiclassical convergence of the scattering profiles.

The analysis is performed in the free-transport frame, a standard approach in Landau damping that isolates the resonant interaction with the background distribution and enables the construction of the scattering map. The finite-time estimate also transfers to the original coordinates, whereas the uniform-in-time statement is naturally formulated in the free-transport frame, as discussed below.
\begin{theorem}[Quantitative finite-time semiclassical convergence]\label{maintheorem1}
Let $\sigma>5/2$ and $T>0$. Assume $\mu\in H^{\sigma+3}_2(\R^3)$ and $K\in L^1(\R^3)$. Let $h(t)$ be a real-valued solution of the nonlinear Vlasov
equation~\eqref{NV} on $[0,T]$ with initial data $h_{\mathrm{in}}$, and define
\[
g(t,z,\xi):=h(t,z+t\xi,\xi).
\]
Let $\delta\in(0,1]$ and, for $\hbar\in(0,\delta]$, let $Q^\hbar(t)$ be a
self-adjoint solution of the nonlinear Hartree equation~\eqref{NH} on $[0,T]$
with initial data $\Qin^\hbar$, and define
\[
P^\hbar(t):=e^{-i\frac{\hbar}{2}t\Delta}Q^\hbar(t)e^{i\frac{\hbar}{2}t\Delta}.
\]
Assume that there exists $M_T>0$ such that one of the following holds:
\begin{enumerate}[label=(\alph*)]
  \item \label{amaintheorem1} $[K]_1<\infty$ and
  \[
  \|g\|_{C([0,T];H^\sigma_2)}+\sup_{\hbar\in(0,\delta]}\|P^\hbar\|_{C([0,T];\mathcal H^{\sigma+3}_2)}\le M_T,
  \]
  \item \label{bmaintheorem1} $[K]_2<\infty$ and
  \[
  \|g\|_{C([0,T];H^{\sigma+3}_2)}+\sup_{\hbar\in(0,\delta]}\|P^\hbar\|_{C([0,T];\mathcal H^{\sigma}_2)}\le M_T.
  \]
\end{enumerate}

Then there exist constants $a_{\sigma,K},b_{\sigma,K}>0$, depending only on
$\sigma$ and $K$, such that, for all $t\in[0,T]$ and $\hbar\in(0,\delta]$,
\begin{align}\label{mainmainconc}
    \norm{g(t)-W^{\hbar}[P^\hbar(t)]}_{H^\sigma_2}
    \le C_1(T)\norm{h_{\rm in}-W^{\hbar}[\Qin^\hbar]}_{H^\sigma_2}
    +C_2(T)\hbar^2,
\end{align}
where
\begin{align}\label{Cconstants}
    \begin{split}
        C_1(T)&:=\exp\sbr{a_{\sigma,K}\br{T\norm{\mu}_{H^{\sigma+1}_2}+T\jb{T}M_T}},\\
        C_2(T)&:=C_1(T)b_{\sigma,K}M_T\br{\norm{\mu}_{H^{\sigma+3}_2}+T\jb{T}M_T}.
    \end{split}
\end{align}
\end{theorem}
The proof of Theorem~\ref{maintheorem1} is given in Section~\ref{Gronwallsection}.

\begin{remark}\label{firstRemark}
The following remarks explain how the convergence rate depends on the initial alignment, identify the origin of the $O(\hbar^2)$ error, clarify the a priori bounds~\ref{amaintheorem1} and~\ref{bmaintheorem1}, and translate the estimate back to the original variables.
\begin{itemize}
\item If the initial data are aligned, $\norm{h_{\rm in}-W^{\hbar}[\Qin^\hbar]}_{H^\sigma_2}=O(\hbar^p)$ for some $p>0$, then
\begin{align*}
    \norm{W^{\hbar}[P^\hbar(t)]- g(t)}_{H^\sigma_2}=O(\hbar^{\min\set{2,p}})\quad\text{for all }t\in[0,T].
\end{align*}
\item The order $O(\hbar^2)$ has a simple structural origin. Formally, the potential commutator corresponds to a centred difference quotient: its leading term is the classical force term, while symmetry cancels the order-$\hbar$ correction, leaving a first non-vanishing remainder of order $\hbar^2$. The assumptions on $K$ allow this remainder and the accompanying nonlinear terms to be controlled in the weighted Sobolev energy estimates leading to the Gr\"onwall argument. The same formal cancellation remains present for Coulomb and gravitational interactions, but the singular low-frequency behaviour of their forces falls outside the Fourier multiplier bounds used here, so the present argument does not apply. Existing strong results for those interactions instead obtain, by different methods, an $O(\hbar)$ rate that is expected to be optimal within their framework~\cite{ChongLaflecheSaffirio2023}.
\item The bounds in \ref{amaintheorem1} and \ref{bmaintheorem1} are assumed a priori; Appendix~\ref{PropagationAppendixReference} provides criteria under which they hold. By Propositions~\ref{prop:vlasovLWP} and~\ref{Hartreeregularityproposition}, the $H^\sigma_2$ Vlasov norm and $\mathcal H^{\sigma+3}_2$ Hartree norm required in \ref{amaintheorem1} (respectively the $H^{\sigma+3}_2$ Vlasov norm and $\mathcal H^\sigma_2$ Hartree norm required in \ref{bmaintheorem1}) are propagated from the initial data provided a lower-order Sobolev norm remains bounded, uniformly in $\hbar$ in the Hartree case. In applications, the high-regularity bounds therefore reduce to assumptions on the initial data together with control of lower-order norms. 
\item Theorem~\ref{maintheorem1} yields an estimate for the full solutions $f=f_\mu+h$ and $\gamma^\hbar=\gamma^\hbar_\mu+Q^\hbar$ in the original variables. By Lemma~\ref{operatorsWigner}~\ref{HartreetoFreeTransport} and~\ref{HomoegeneousstatesandWignertransform},
\[
W^{\hbar}[\gamma^\hbar(t)](x,\xi)=\mu(\xi)+W^{\hbar}[P^\hbar(t)](x-t\xi,\xi),
\]
and hence, with $\Phi_t(x,\xi):=(x-t\xi,\xi)$,
\[
f(t)-W^{\hbar}[\gamma^\hbar(t)]
   =\bigl(g(t)-W^{\hbar}[P^\hbar(t)]\bigr)\circ\Phi_t.
\]
Since $\Phi_t$ preserves the measure, the $L^2$ norms of the two differences coincide, and the weighted norms satisfy $\norm{u\circ\Phi_t}_{H^\sigma_2}\lesssim_\sigma\jb{t}^\sigma\norm{u}_{H^\sigma_2}$ by Lemma~\ref{japanesebracketlemma}~\ref{JBshear}. Consequently~\eqref{mainmainconc} holds for $f(t)-W^{\hbar}[\gamma^\hbar(t)]$ with the same rate $O(\hbar^2)$ and constants enlarged by $\jb{t}^\sigma$.
\end{itemize}
\end{remark}

The constants in~\eqref{mainmainconc} are independent of $\hbar$ for every fixed time horizon $T$, but grow rapidly with $T$, so Theorem~\ref{maintheorem1} gives no direct information about the limit $t\to\infty$. The required large-time input is supplied entirely by the quantum dynamics: in previous work~\cite{smith2025phasemixinghartreeequation}, phase mixing and scattering for the Hartree equation near Penrose-stable homogeneous states were established uniformly in $\hbar$, with the framed solutions $P^\hbar(t)$ converging to scattering profiles $Q^\hbar_\infty$ at rate $\jb{t}^{-3/2}$ and their higher Sobolev norms growing at most polynomially.

Combined with Theorem~\ref{maintheorem1}, these uniform quantum estimates transfer to the Vlasov equation: fixed-time semiclassical convergence carries the polynomial norm bounds to the classical solution, yielding global existence by continuation, and uniform quantum scattering shows that $g(t)$~is Cauchy in~$H^\sigma_2$ and hence scatters to a profile $h_\infty$, without any input from the classical damping theory. This transfer is recorded in Proposition~\ref{Vlasovscattering}. The finite-time and scattering estimates can then be balanced: up to polynomial factors in $T$, the semiclassical error on $[0,T]$ for data aligned at order $O(\hbar^p)$ is bounded by $\hbar^{\min\set{2,p}}e^{AT^{9/2}}$, whereas both solutions are within $O(T^{-3/2})$ of their scattering profiles after time $T$, and choosing the horizon $T=T(\hbar)$ yields the uniform-in-time logarithmic rate stated in Theorem~\ref{theoremonapplication}.

\begin{theorem}[Uniform-in-time semiclassical convergence]\label{theoremonapplication}
Let $\sigma>5/2$ and assume $K\in L^1(\R^3)$ with $[K]_2<\infty$. There exist exponents $\overline\sigma>\sigma_0>\sigma$ with the following property. Let $\mu\in H^{\overline\sigma}_4(\R^3)$ be non-negative and suppose that the pair $(K,\mu)$ satisfies the uniform Penrose condition~\eqref{Penrose2} with constants $\kappa>0$ and $\delta\in(0,1]$. Then there exists $\varepsilon_0>0$ such that the following holds.

Let $\hin\in H^\sigma_2$ and, for each $\hbar\in(0,\delta]$, let $Q^\hbar_{\mathrm{in}}$ be a self-adjoint semiclassical initial perturbation. Assume the following:
\begin{enumerate}[label=(\roman*)]
\item (Smallness) the initial perturbations satisfy
\begin{align}\label{initialconds}
\sup_{\hbar\in(0,\delta]}\sum_{|\alpha|\le2}\|\bm{x}^\alpha Q^\hbar_{\mathrm{in}}\|_{\mathcal H^{\sigma_0}_2}\le\ep_0,
\end{align}
\item\label{alignmentbullet} (Alignment) there exist $p>0$ and $C_0>0$ such that, for every $\hbar\in(0,\delta]$,
\begin{align*}
\norm{W^{\hbar}[Q^\hbar_{\mathrm{in}}]-\hin}_{H^\sigma_2}
\leq C_0\hbar^p.
\end{align*}
\end{enumerate}
Then the corresponding Vlasov and Hartree solutions exist globally and admit scattering profiles
\begin{align*}
h_\infty
:=
\lim_{t\to\infty}g(t)
\quad\text{in }H^\sigma_2,
\qquad
Q^\hbar_\infty
:=
\lim_{t\to\infty}P^\hbar(t)
\quad\text{in }\mathcal H^\sigma_2.
\end{align*}
Moreover, there exist $\delta_*\in(0,\delta]$ and a constant $C_\infty=C_\infty(\sigma,\mu,K,\kappa,C_0,p)>0$
such that, adopting the conventions $g(\infty):=h_\infty$ and $P^\hbar(\infty):=Q^\hbar_\infty$, for every $t\in[0,\infty]$ and $\hbar\in(0,\delta_*]$,
\begin{align}\label{uniformrate}
\norm{W^{\hbar}[P^\hbar(t)]-g(t)}_{H^\sigma_2}
\leq
C_\infty\br{\ln\frac{1}{\hbar}}^{-1/3}.
\end{align}
\end{theorem}
The proof of Theorem~\ref{theoremonapplication} is given in Section~\ref{uniformtimesection}.
\begin{remark}[Uniformity and convergence rate]
Unlike Theorem~\ref{maintheorem1}, which gives a conditional comparison on a finite interval $[0,T]$ over which the required a priori bounds hold, Theorem~\ref{theoremonapplication} proves global existence of both solutions and the existence of their scattering profiles. Estimate~\eqref{uniformrate} controls the semiclassical error uniformly over the entire evolution and at the scattering limit $t=\infty$, where it gives strong convergence of the quantum scattering profiles to the classical scattering profile. The logarithmic rate results from balancing the growth of the finite-time estimate against the $\jb{t}^{-3/2}$ scattering decay; no optimality is claimed.
\end{remark}

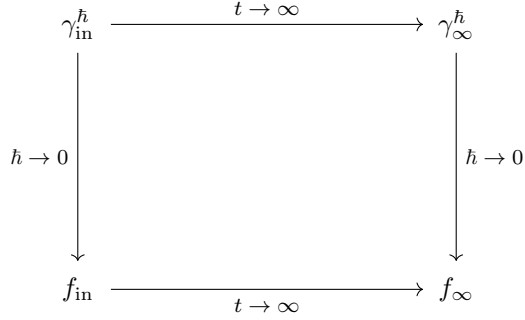
\begin{figure}[H]
\vspace{0cm}
\centering
\begin{tikzpicture}[domain=0:5,xscale=0.25,yscale=0.25]
\draw [->] (1.75,0)--(18.25,0);
\draw [->] (1.75,14)--(18.25,14);
\draw [->] (0,12.5)--(0,1.5);
\draw [->] (20,12.5)--(20,1.5);
\node at (0,0) {$f_{\mathrm{in}}$};
\node at (20,0) {$f_{\infty}$};
\node at (20,14) {$\gamma^\hbar_{\infty}$};
\node at (0,14) {$\gamma^\hbar_{\mathrm{in}}$};
\node [below] at (10,0) {\footnotesize{$t\to\infty$}};
\node [above] at (10,14) {\footnotesize{$t\to\infty$}};
\node [left] at (0,7) {\footnotesize{$\hbar\to0$}};
\node [right] at (20,7) {\footnotesize{$\hbar\to0$}};
\end{tikzpicture}
\captionsetup{width=.7\linewidth}
\caption{Semiclassical limit and scattering at positive density. The horizontal limits are taken along the free-transport flow, and the vertical limits at the level of Wigner transforms.}
\label{map}
\end{figure}

The scheme is summarised in Figure~\ref{map}. Since the homogeneous steady states are invariant under the free dynamics, the full solutions, transported along the free flow, differ from them exactly by the framed perturbations:
\begin{align*}
f(t)\circ\Phi_t^{-1}=f_\mu+g(t)
\qaq
e^{-i\frac{\hbar}{2}t\Delta}\gamma^\hbar(t)e^{i\frac{\hbar}{2}t\Delta}
=\gamma^\hbar_\mu+P^\hbar(t).
\end{align*}
With $f_\infty:=f_\mu+h_\infty$ and
$\gamma^\hbar_\infty:=\gamma^\hbar_\mu+Q^\hbar_\infty$, the two horizontal
arrows of the figure abbreviate the long-time convergence of the framed
solutions,
\begin{align*}
e^{-i\frac{\hbar}{2}t\Delta}\gamma^\hbar(t)e^{i\frac{\hbar}{2}t\Delta}
\rightarrow\gamma^\hbar_\infty
\qaq
f(t)\circ\Phi_t^{-1}\rightarrow f_\infty
\qquad\text{as }t\to\infty,
\end{align*}
while the two vertical arrows are semiclassical limits of Wigner transforms, in which the steady states cancel because
$W^{\hbar}[\gamma^\hbar_\mu]=\mu$
(Lemma~\ref{operatorsWigner}~\ref{HomoegeneousstatesandWignertransform}). Two of the four arrows are given. Quantum scattering, the arrow $\gamma^\hbar_{\mathrm{in}}\to\gamma^\hbar_\infty$ along the top, holds uniformly in $\hbar$ by~\cite{smith2025phasemixinghartreeequation}; convergence of the initial data, $W^{\hbar}[\gamma^\hbar_{\mathrm{in}}]\to f_{\mathrm{in}}$ on the left, is an assumption. The remaining two arrows are deduced. Classical scattering, the arrow $f_{\mathrm{in}}\to f_\infty$ along the bottom, coincides with the conclusion of Bedrossian, Masmoudi and Mouhot~\cite{BedrossianMasmoudiMouhot2018}, but enters as an output rather than an input: Proposition~\ref{Vlasovscattering} derives it from quantum scattering and the convergence of the initial data, through the finite-time estimate of Theorem~\ref{maintheorem1}. Finally, convergence of the scattering states, $W^{\hbar}[\gamma^\hbar_\infty]\to f_\infty$ on the right, is the new conclusion of Theorem~\ref{theoremonapplication}, which yields strong convergence in weighted Sobolev norms with an explicit rate and controls the distance between the two evolutions uniformly for all $t\in[0,\infty]$.

\subsection{Existing results}\label{existingresultssection}
This subsection recalls the results most directly related to the two ingredients of the paper: large-time dynamics near homogeneous states and semiclassical convergence from Hartree to Vlasov. 

\subsubsection{Large-time dynamics}

Landau damping is a phase-mixing phenomenon in which macroscopic observables decay near homogeneous steady states without dissipation. Landau first identified the phenomenon at the linearised level for the Vlasov equation~\cite{Landaupaper1946}. The first proof of nonlinear Landau damping was obtained in analytic regularity~\cite{mou-vil-2011}, and the subject has since been developed in several directions; see, for example,~\cite{bed-mas-mou-2016,BedrossianMasmoudiMouhot2018,GrenierNguyenRodnianski2021,BedrossianMasmoudiMouhot2022,IonescuPausaderWangWidmayer2024Poisson}. In particular, damping and scattering were proved in weighted Sobolev spaces on the whole space, for localised perturbations of an infinite homogeneous background and integrable screened interactions, including the screened Coulomb kernel~\cite{BedrossianMasmoudiMouhot2018}.

Analogous asymptotic stability has been studied for the Hartree equation. Lewin and Sabin~\cite{LewinSabinII2014} proved weak convergence to homogeneous equilibria in two dimensions, and the analysis was extended to dimensions $d\geq3$, with scattering near positive-temperature equilibria~\cite{ChenHongPavlovic2018}. Scattering for initial perturbations at critical regularity was established in~\cite{CollotdeSuzzoni2020,CollotdeSuzzoni2022,BorieHadamaSabin2025}. With more regular initial perturbations, the large-time behaviour has been quantified through pointwise decay and refined scattering estimates. For the Coulomb interaction, which lies outside the short-range Penrose setting, the linearised dynamics near homogeneous equilibria were analysed in~\cite{NguyenYou2023}, followed by decay and modified scattering for the nonlinear equation near vacuum~\cite{NguyenYou2024}. For short-range kernels, You~\cite{You2024} established phase-mixing estimates and Hilbert--Schmidt scattering at fixed $\hbar$, and nonlinear Landau damping and scattering were recently proved for the Hartree--Fock equation in the presence of a sufficiently small exchange interaction~\cite{NguyenYou2026HartreeFock}.

The author's earlier work~\cite{smith2025phasemixinghartreeequation} established pointwise phase-mixing decay and scattering in weighted quantum Sobolev spaces near Penrose-stable homogeneous states, with estimates uniform in $\hbar$; the bounds required here are recalled in Theorem~\ref{Smithresult}. This uniformity permits the quantum scattering theory to be transferred to the Vlasov equation through the finite-time semiclassical limit and the comparison of the two dynamics to be extended uniformly in time. In particular, Proposition~\ref{Vlasovscattering} recovers through the limit a Sobolev-space Vlasov scattering result of the type proved by Bedrossian, Masmoudi and Mouhot~\cite{BedrossianMasmoudiMouhot2018}, with hypotheses inherited from the quantum theory.
\subsubsection{Semiclassical convergence and scattering profiles}
The semiclassical limit from Hartree to Vlasov was first established through weak convergence of Wigner transforms~\cite{LionsPaul1993,MarkowichMauser1993}. Quantitative strong convergence was first obtained in Hilbert--Schmidt and trace topologies for regular interactions~\cite{AthanassoulisPaulPezzottiPulvirenti2011,BenedikterPortaSaffirioSchlein2016}, and subsequently sharpened to the optimal $O(\hbar)$ rate and extended to singular interactions, including the Coulomb potential~\cite{LaflecheSaffirio2023,ChongLaflecheSaffirio2023}. These results concern finite-mass states and fixed time intervals. The large-time semiclassical limit near vacuum was recently studied by Hadama and Hong~\cite{HadamaHong2025semiclassicallimitquantumscattering}, for short-range interactions including singular inverse-power laws. They obtained quantum dispersion estimates uniform in $\hbar$, together with scattering at each fixed $\hbar$ under a smallness condition independent of $\hbar$, and proved that, along subsequences, the Wigner transforms of the quantum scattering profiles converge weakly in $L^2$ to a Vlasov scattering profile.

At positive density, the semiclassical limit appears to have been addressed only by Lewin and Sabin~\cite{LewinSabin2020}, who proved weak convergence of Wigner transforms in the relative-energy setting, without a rate, and used it to construct global Vlasov solutions. Strong semiclassical convergence near translation-invariant steady states was raised as an open problem in the author's earlier work~\cite[Section~1.7]{smith2025phasemixinghartreeequation}, where the uniform-in-$\hbar$ scattering theory was proposed as a framework for recovering classical scattering data through the limit. This paper resolves that problem for the regular kernels considered here: Theorem~\ref{maintheorem1} provides the strong quantitative comparison, with the $O(\hbar^2)$ rate of Remark~\ref{firstRemark}, and Theorem~\ref{theoremonapplication} extends it, for aligned data near Penrose-stable states, uniformly in time in the free-transport frame, including the convergence of the scattering profiles.

\subsection{Outline of proof}
Section~\ref{movingframeformulation} rewrites both equations in the frame of the free-transport flow, where they take the same structural form and their difference can be tracked directly. Section~\ref{energyestimatessection} supplies the two ingredients of the comparison: energy estimates in weighted Sobolev spaces for the shared linear and bilinear structure, and $O(\hbar^2)$ bounds on the operators encoding the semiclassical discrepancy. Section~\ref{Gronwallsection} combines these in a Gr\"onwall argument, proving Theorem~\ref{maintheorem1}. Section~\ref{scatteringproofsection} transfers the uniform-in-$\hbar$ scattering theory of~\cite{smith2025phasemixinghartreeequation} to the Vlasov equation through the finite-time semiclassical limit, proving Proposition~\ref{Vlasovscattering}, and Section~\ref{uniformtimesection} balances the finite-time estimate against the scattering decay, with a time horizon chosen as a function of $\hbar$, proving Theorem~\ref{theoremonapplication}.

\section{Moving-frame formulation}\label{movingframeformulation}
Both equations are now rewritten in frames adapted to phase mixing near the homogeneous steady states: the Vlasov perturbation is composed with the free-transport flow, while the Hartree perturbation is conjugated by the free Schr\"odinger group, so that its Wigner transform is expressed in the free-transport frame. Throughout this section, $\Phi_t(x,\xi):=(x-t\xi,\xi)$ denotes the free-transport map.

\subsection{Vlasov equation in the moving frame} The Vlasov perturbation is expressed in free-transport variables by introducing
\begin{align*}
    g(t,z,\xi):=h(t,z+t\xi,\xi),
\end{align*}where $h$ denotes the perturbation evolving according to \eqref{NV}; equivalently, $h(t)=g(t)\circ\Phi_t$, so that $\rho_h=\rho_{g\circ\Phi_t}$ and the force field is $\nabla V_g$, where
\begin{align}\label{selfconsistentpotential}
    V_f(t,x):=\br{K*_x\rho_{f(t)\circ\Phi_t}}(x).
\end{align}Then $g$ satisfies
\begin{align}
    \begin{mycases}
    \p_tg=L_\mu[g]+N[g,g],\label{equationforg(t)}\\
        g(0)=\hin,
    \end{mycases}
\end{align}where the operators $L\colon (\mu,g)\mapsto L_\mu[g]$ and $N\colon (g_1,g_2)\mapsto N[g_1,g_2]$ are defined by\begin{align*}
    L_\mu[g](t,z,\xi)&:=\nabla V_g(t,z+t\xi)\cdot\nabla_\xi\mu(\xi),\\
    N[g_1,g_2](t,z,\xi)&:=\nabla V_{g_1}(t,z+t\xi)\cdot\br{\nabla_\xi-t\nabla_z}g_2(t,z,\xi).
\end{align*}

The moving-frame analysis is carried out in the Fourier variables $(k,\eta)$, in which the free flow acts by explicit shears; the following elementary identities record this action.
\begin{lemma}[Fourier identities in the frame]\label{framefourierlemma}
Let $F,f\in L^1(\R^3\times\R^3)$ and let $u$ satisfy $\widehat u\in L^1(\R^3)$. Then, for every $t\in\R$:
\begin{enumerate}[label=(\roman*)]
    \item \label{shearidentity} $\mathcal F_{z,\xi}\sbr{F(z+t\xi,\xi)}(k,\eta)=\widehat F(k,\eta-kt)$.
    \item \label{framedensityidentity} $\widehat{\rho_{f\circ\Phi_t}}(k)=\widehat f(k,k t)$.
    \item \label{frameproductidentity} $\displaystyle\mathcal F_{z,\xi}\sbr{u(z+t\xi)f(z,\xi)}(k,\eta)=(2\pi)^{-3}\int_{\R^3}\widehat u(\ell)\widehat f(k-\ell,\eta-t\ell)\,d\ell$.
\end{enumerate}
\end{lemma}
\begin{proof}
For~\ref{shearidentity}, substitute $x=z+t\xi$ in~\eqref{FTconvention}, so that $e^{-ik\cdot z-i\eta\cdot\xi}=e^{-ik\cdot x-i(\eta-kt)\cdot\xi}$. For~\ref{framedensityidentity}, the substitution $z=x-t\xi$ gives $\widehat{\rho_{f\circ\Phi_t}}(k)=\int_{\R^3}\int_{\R^3}f(z,\xi)e^{-ik\cdot(z+t\xi)}\,dz d\xi=\widehat f(k,k t)$. For~\ref{frameproductidentity}, expand $u$ by Fourier inversion and interchange the integrals: each mode contributes $\mathcal F_{z,\xi}\sbr{e^{i\ell\cdot(z+t\xi)}f}(k,\eta)=\widehat f(k-\ell,\eta-t\ell)$.
\end{proof}
Item~\ref{shearidentity} is the phase-mixing mechanism in Fourier variables: composition with the free flow shears the velocity frequency along $kt$. Item~\ref{framedensityidentity} expresses the density of a framed solution as the trace of its Fourier transform on the moving line $\eta=k t$, and implies that the Fourier transform of~\eqref{selfconsistentpotential} is \begin{align}
    \widehat{V_f}(t,k)=\widehat K(k)\widehat f(t,k,k t).\label{Vfourier}
\end{align}

The Fourier transform of $g$ satisfies the equation\begin{align*}
    \p_t\widehat g=\mathcal L_\mu[g]+\mathcal N[g,g]
\end{align*}
where the operators $\mathcal L\colon (\mu,g)\mapsto \mathcal L_\mu[g]$ and $\mathcal N\colon (g_1,g_2)\mapsto \mathcal N[g_1,g_2]$ are defined by\begin{align*}
    \mathcal L_\mu[g](t,k,\eta):=\mathcal F_{z,\xi}\sbr{L_\mu[g](t,z,\xi)}(k,\eta)\qaq \mathcal N[g_1,g_2](t,k,\eta):=\mathcal F_{z,\xi}\sbr{N[g_1,g_2](t,z,\xi)}(k,\eta)
\end{align*}and satisfy
\begin{align}
   \mathcal L_\mu[g](t,k,\eta)&=-\widehat K(k)\widehat g(t,k,kt) k\cdot (\eta-kt)\widehat{\mu}(\eta-kt),\label{mathcalL}\\
    \mathcal N[g_1,g_2](t,k,\eta)&=-(2\pi)^{-3}\int_{\R^3} \widehat K(\ell)\widehat {g_1}(t,\ell,\ell t)\ell\cdot (\eta-kt)\widehat {g_2}(t,k-\ell,\eta-t\ell)\,d\ell\label{mathcalN}.
\end{align}
Indeed, the linear term has the form $F(z+t\xi,\xi)$ with $F(x,\xi)=\nabla V_g(t,x)\cdot\nabla_\xi\mu(\xi)$, whose Fourier transform factorises as $\widehat F(k,\eta)=-k\cdot \eta\widehat{V_g}(t,k)\widehat\mu(\eta)$, so \eqref{mathcalL} follows from Lemma~\ref{framefourierlemma}~\ref{shearidentity} and \eqref{Vfourier}. For the nonlinear term, \eqref{mathcalN} follows from Lemma~\ref{framefourierlemma}~\ref{frameproductidentity} applied componentwise with $u=\p_jV_{g_1}$ and $f=(\p_{\xi_j}-t\p_{z_j})g_2$, since $(\eta-t\ell)-t(k-\ell)=\eta-kt$.

\subsection{Hartree equation in the moving frame}An analogous change of variables is performed for the Hartree equation at the level of density matrices, and the associated Wigner formulation is recorded.

For the Hartree perturbation equation \eqref{NH}, the free Schr\"odinger evolution is factored out by introducing $P^\hbar(t):=e^{-i\frac{\hbar}{2}t\Delta}Q^\hbar(t)e^{i\frac{\hbar}{2}t\Delta}$, with $P^\hbar(0)=\Qin^\hbar$. The conjugation absorbs the kinetic commutator in \eqref{NH}, so that
\begin{align}\label{framedvonneumann}
    i\hbar\p_tP^\hbar(t)=e^{-i\frac{\hbar}{2}t\Delta}\sbr{K*_x\rho_{Q^\hbar}(t),Q^\hbar(t)}e^{i\frac{\hbar}{2}t\Delta}+e^{-i\frac{\hbar}{2}t\Delta}\sbr{K*_x\rho_{Q^\hbar}(t),\gamma^\hbar_\mu}e^{i\frac{\hbar}{2}t\Delta}.
\end{align}
The passage to the Wigner formulation requires only the following commutator identity.

\begin{lemma}[Wigner transform of a potential commutator]\label{potentialcommutatorlemma}
Let $V$ satisfy $\widehat V\in L^1(\R^3)$ and let $\gamma\in\mathcal L^2$. Then $\sbr{V,\gamma}\in\mathcal L^2$ and
\begin{align}\label{Wignercommutatoridentity}
    W^{\hbar}\sbr{V,\gamma}(x,\xi)=(2\pi)^{-3}\int_{\R^3}\widehat V(\ell)e^{i\ell\cdot x}\br{W^{\hbar}[\gamma]\br{x,\xi-\tfrac{\hbar}{2}\ell}-W^{\hbar}[\gamma]\br{x,\xi+\tfrac{\hbar}{2}\ell}}\,d\ell
\end{align}
in $L^2(\R^3\times\R^3)$. The same identity holds for the homogeneous state $\gamma^\hbar_\mu$ when $\mu,\widehat\mu\in L^1(\R^3)$, with $W^{\hbar}[\gamma^\hbar_\mu]=\mu$.
\end{lemma}
\begin{proof}
Suppose first that $\gamma$ has a Schwartz kernel. The integral kernels of $V\gamma$ and $\gamma V$ are $V(x_1)\gamma(x_1,x_2)$ and $\gamma(x_1,x_2)V(x_2)$, so Definition~\ref{Wignerdefinition} gives
\begin{align*}
    W^{\hbar}\sbr{V,\gamma}(x,\xi)=(2\pi\hbar)^{-3}\int_{\R^3}e^{-i\xi\cdot y/\hbar}\br{V\br{x+\tfrac{y}{2}}-V\br{x-\tfrac{y}{2}}}\gamma\br{x+\tfrac{y}{2},x-\tfrac{y}{2}}\,dy.
\end{align*}
Expanding $V$ by Fourier inversion and combining the phases through
\begin{align*}
    e^{-i\xi\cdot y/\hbar}e^{\pm i\ell\cdot y/2}=e^{-i(\xi\mp\frac{\hbar}{2}\ell)\cdot y/\hbar},
\end{align*}
the two $y$-integrals are, by Definition~\ref{Wignerdefinition}, the Wigner transforms $W^{\hbar}[\gamma]\br{x,\xi\mp\tfrac{\hbar}{2}\ell}$, and \eqref{Wignercommutatoridentity} follows upon interchanging the $\ell$ and $y$ integrals.

Since $\widehat V\in L^1$, the multiplication operator $V$ is bounded, and hence $\gamma\mapsto\sbr{V,\gamma}$ is continuous on $\mathcal L^2$. The right-hand side of~\eqref{Wignercommutatoridentity} is also continuous in $W^{\hbar}[\gamma]\in L^2$, since translations and modulations preserve the $L^2$ norm and, by Minkowski's integral inequality,
\begin{align*}
    \norm{\int_{\R^3}\widehat V(\ell)e^{i\ell\cdot x}\br{W^{\hbar}[\gamma]\br{x,\xi-\tfrac{\hbar}{2}\ell}-W^{\hbar}[\gamma]\br{x,\xi+\tfrac{\hbar}{2}\ell}}\,d\ell}_{L^2_{x,\xi}}\leq2\norm{\widehat V}_{L^1}\norm{W^{\hbar}[\gamma]}_{L^2}.
\end{align*}
The identity therefore extends from Schwartz kernels to every $\gamma\in\mathcal L^2$ by density. 

For $\gamma=\gamma^\hbar_\mu$, the kernel~\eqref{intkernelofsteadystate} equals $\widehat\mu\br{-(x-y)/\hbar}$ and is therefore integrable in $x-y$, so the same computation applies, the $y$-integrals evaluating to $\mu\br{\xi\mp\tfrac{\hbar}{2}\ell}$ by Fourier inversion and all integrals converging absolutely since $\norm{\mu}_{L^\infty}\leq(2\pi)^{-3}\norm{\widehat\mu}_{L^1}$.
\end{proof}

Define $g^\hbar(t):=W^{\hbar}[P^\hbar(t)]$. By Lemma~\ref{operatorsWigner}~\ref{HartreetoFreeTransport},
\begin{align}\label{WignerQfromg}
    W^{\hbar}[Q^\hbar(t)](x,\xi)=g^\hbar(t,x-t\xi,\xi),
\end{align}
and hence, by the density identity of Lemma~\ref{operatorsWigner}~\ref{densityproperty}, the multiplication potential in \eqref{framedvonneumann} is $K*_x\rho_{Q^\hbar(t)}=K*_x\rho_{g^\hbar(t)\circ\Phi_t}=V_{g^\hbar}(t)$.

The Wigner transform of \eqref{framedvonneumann} is now computed. First, Lemma~\ref{operatorsWigner}~\ref{HartreetoFreeTransport} removes the conjugations: for either commutator,
\begin{align*}
    W^{\hbar}\sbr{e^{-i\frac{\hbar}{2}t\Delta}\sbr{V_{g^\hbar}(t),\gamma}e^{i\frac{\hbar}{2}t\Delta}}(z,\xi)=W^{\hbar}\sbr{V_{g^\hbar}(t),\gamma}(z+t\xi,\xi),
    \qquad\gamma\in\set{Q^\hbar(t),\gamma^\hbar_\mu}.
\end{align*}
The right-hand sides are given by Lemma~\ref{potentialcommutatorlemma} at $x=z+t\xi$: the shifted Wigner transforms there are $\mu\br{\xi\mp\tfrac{\hbar}{2}\ell}$ for the homogeneous state, while for $Q^\hbar$, by \eqref{WignerQfromg}, the frame shifts both arguments,
\begin{align*}
    W^{\hbar}[Q^\hbar(t)]\br{z+t\xi,\xi\mp\tfrac{\hbar}{2}\ell}=g^\hbar\br{t,z\pm\tfrac{\hbar t}{2}\ell,\xi\mp\tfrac{\hbar}{2}\ell}.
\end{align*}
It follows, after dividing by $i\hbar$ and inserting the Fourier representation~\eqref{Vfourier} of the self-consistent field, that $g^\hbar$ satisfies
\begin{align}
\label{eqforg^hbar(t)}\begin{mycases}
    \p_tg^\hbar=L^\hbar_\mu[g^\hbar]+N^\hbar[g^\hbar,g^\hbar],\\
g^\hbar(0)=W^\hbar[\Qin^\hbar],
\end{mycases}
\end{align}
where the operators $L^\hbar\colon (\mu,g)\mapsto L^\hbar_\mu[g]$ and $N^\hbar\colon (g_1,g_2)\mapsto N^\hbar[g_1,g_2]$ are defined by
\begin{align*}
    L^\hbar_\mu[g](t,z,\xi)&:=\frac{(2\pi)^{-3}}{i\hbar}\int_{\R^3}\widehat K(\ell)\widehat g(t,\ell,\ell t)e^{i\ell\cdot(z+t\xi)}\br{\mu\br{\xi-\tfrac{\hbar}{2}\ell}-\mu\br{\xi+\tfrac{\hbar}{2}\ell}}\,d\ell,\\
    N^\hbar[g_1,g_2](t,z,\xi)&:=\frac{(2\pi)^{-3}}{i\hbar}\int_{\R^3}\widehat K(\ell)\widehat {g_1}(t,\ell,\ell t)e^{i\ell\cdot(z+t\xi)}\br{g_2\br{t,z+\tfrac{\hbar t}{2}\ell,\xi-\tfrac{\hbar}{2}\ell}-g_2\br{t,z-\tfrac{\hbar t}{2}\ell,\xi+\tfrac{\hbar}{2}\ell}}\,d\ell.
\end{align*}

The Fourier transform of $g^\hbar$ satisfies the equation
\begin{align*}
    \p_t\widehat{g^\hbar}(t,k,\eta)=\mathcal L^\hbar_\mu[g^\hbar](t,k,\eta)+\mathcal N^\hbar[g^\hbar,g^\hbar](t,k,\eta),
\end{align*}where the operators 
$\mathcal L^\hbar\colon(\mu,g)\mapsto \mathcal L^\hbar_{\mu}[g]$ and $\mathcal N^\hbar\colon (g_1,g_2)\mapsto \mathcal N^\hbar[g_1,g_2]$ are defined by\begin{align*}
    \mathcal L^\hbar_\mu[g](t,k,\eta):=\mathcal F_{z,\xi}\sbr{L^\hbar_\mu[g](t,z,\xi)}(k,\eta)\qaq \mathcal N^\hbar[g_1,g_2](t,k,\eta):=\mathcal F_{z,\xi}\sbr{N^\hbar[g_1,g_2](t,z,\xi)}(k,\eta)
\end{align*}and satisfy
\begin{align}
    \mathcal L^\hbar_{\mu}[g](t,k,\eta)&=-2\hbar^{-1}\widehat K(k)\widehat {g}(t,k,kt)\sin\br{\mytextfrac{\hbar}{2}k\cdot(\eta-kt)}\widehat \mu(\eta-kt),\label{mathcalLhbar}\\
    \mathcal N^\hbar[g_1,g_2](t,k,\eta)&=-2\hbar^{-1}(2\pi)^{-3}\int_{\R^3} \widehat{K}(\ell )\widehat{g_1}(t,\ell,\ell t)\sin\br{\mytextfrac{\hbar}{2}\ell \cdot(\eta-kt)}\widehat{g_2}(t,k-\ell ,\eta-\ell t)\,d\ell.\label{mathcalNhbar}
\end{align}
Indeed, the modulation and translations combine into
\begin{align*}
    \mathcal F_{z,\xi}\sbr{e^{i\ell\cdot(z+t\xi)}f\br{z\pm\tfrac{\hbar t}{2}\ell,\xi\mp\tfrac{\hbar}{2}\ell}}(k,\eta)=e^{\mp i\frac{\hbar}{2}\ell\cdot(\eta-kt)}\widehat f(k-\ell,\eta-t\ell),
\end{align*}
which, applied with $f=g_2(t)$ and combined with $e^{-i\theta}-e^{i\theta}=-2i\sin\theta$ at $\theta=\tfrac{\hbar}{2}\ell\cdot(\eta-kt)$, yields \eqref{mathcalNhbar}. In $L^\hbar_\mu$, the bracket has no $z$-dependence, so the Fourier transform in $z$ collapses the $\ell$-integral at $\ell=k$ and cancels the factor $(2\pi)^{-3}$; the same computation in the $\xi$ variable alone then yields \eqref{mathcalLhbar}.

As $\hbar\to0$, the multiplier $2\hbar^{-1}\sin\br{\mytextfrac{\hbar}{2}q\cdot(\eta-kt)}$ converges to $q\cdot(\eta-kt)$, with $q=k$ in the linear term and $q=\ell$ in the nonlinear term, so that \eqref{mathcalLhbar} and \eqref{mathcalNhbar} formally recover the Vlasov operators \eqref{mathcalL} and \eqref{mathcalN}; the quantitative version of this convergence is the content of Section~\ref{quantumerrorsection}.

\subsection{Difference evolution}The difference between the classical and quantum dynamics is expressed in terms of linear and nonlinear error operators. 

Define $R^L\colon (\mu,g)\mapsto R^L_\mu[g]$ and $R^N\colon (g_1,g_2)\mapsto R^N[g_1,g_2]$ by
\begin{align}
    R^L_\mu[g]:=L_\mu[g]-L^\hbar_\mu[g]\qaq R^N[g_1,g_2]:=N[g_1,g_2]-N^\hbar[g_1,g_2],\label{defnsofquantum}
\end{align}
which measure the deviation of solutions under the classical dynamics from the quantum dynamics.

The difference in evolution of the solution $g$ to the Vlasov equation~\eqref{equationforg(t)} with the Hartree equation~\eqref{eqforg^hbar(t)} follows
\begin{align*}
    \p_t\br{g(t)-g^\hbar(t)}=L_\mu[g]-L^\hbar_\mu[g^\hbar]+N[g,g]-N^\hbar[g^\hbar,g^\hbar].
\end{align*}
Using the linearity of $L_\mu$, $L^\hbar_\mu$ and $R_\mu^L$ and the bilinearity of $N$, $N^\hbar$ and $R^N$, two natural decompositions arise: \begin{align}\label{Twoframes}
    \p_t\br{g(t)-g^\hbar(t)}=\begin{mycases}
        L_\mu[g-g^\hbar](t)+ N[g,g-g^\hbar](t)+ N[g-g^\hbar,g^\hbar](t)+ R^{L}_\mu[g^\hbar](t)+ R^{N}[g^\hbar,g^\hbar](t),\\
        L^\hbar_\mu[g-g^\hbar](t)+N^{\hbar}[g^\hbar,g-g^\hbar](t)+N^\hbar[g-g^\hbar,g](t)+ R^{ L}_\mu[g](t)+R^{ N}[g,g](t).
    \end{mycases} 
\end{align}The first form places the semiclassical remainders on $g^\hbar$ and is used in case~\ref{amaintheorem1}, where higher regularity is assumed for the Hartree solution. The second places them on $g$ and is used in case~\ref{bmaintheorem1}, where higher regularity is assumed for the Vlasov solution.

\section{Energy and semiclassical error estimates}\label{energyestimatessection}
Energy estimates are derived in weighted Sobolev spaces for the operators appearing in the moving-frame equations~\eqref{equationforg(t)} and~\eqref{eqforg^hbar(t)}, and the error operators encoding the semiclassical discrepancy~\eqref{defnsofquantum}.

\subsection{Fourier--Sobolev preliminaries}Auxiliary Fourier--Sobolev estimates adapted to the free-transport frame are collected. Throughout, $A\lesssim B$ means $A\leq CB$ for some $C>0$ independent of $\hbar>0$ and $t\geq0$, and $A\lesssim_{\sigma,s,K}B$ additionally indicates that $C$ depends only on the subscripted parameters; dependence on $K$ is always through the finite seminorm $[K]_1$ or $[K]_2$ assumed in the statement at hand.

The weighted space $H^\sigma_2$ is chosen to ensure control of Fourier-side Sobolev embeddings, as quantified in the following lemma.
\begin{lemma}\label{boundonenergy}
    Let $\sigma\geq0$ and let $w\in H^\sigma_2(\R^3\times\R^3)$. Then for any $t\geq0$,
    \begin{align*}\norm{\jb{k,kt}^\sigma \widehat{w}(k,kt)}_{L^2_k}&\lesssim_\sigma\norm{w}_{H^\sigma_2}.
    \end{align*}
\end{lemma}
\begin{proof}
Fix $k\in\R^3$. The Sobolev embedding $H^2(\R^3)\hookrightarrow L^\infty(\R^3)$ in the $\eta$ variable, together with the Leibniz rule and the symbol bound of Lemma~\ref{japanesebracketlemma}~\ref{JBsymbol}, gives
\begin{align*}
    \abs{\jb{k,kt}^\sigma\widehat w(k,kt)}\leq\sup_\eta\abs{\jb{k,\eta}^\sigma \widehat{w}(k,\eta)}\lesssim_\sigma\sum_{\aal\leq 2}\norm{\jb{k,\cdot}^{\sigma}D^\alpha_\eta\widehat w(k,\cdot)}_{L^2_{\eta}}.
\end{align*}
Taking the $L^2_k$ norm and applying Plancherel's theorem, with $D^\alpha_\eta\widehat w=\mathcal F\sbr{(-i\xi)^\alpha w}$, bounds the left-hand side by a constant multiple of $\sum_{\aal\leq2}\norm{\xi^\alpha w}_{H^\sigma}\lesssim\norm{w}_{H^\sigma_2}$.
\end{proof}

\begin{lemma}Let $s>3/2$. \label{importantlemma}
\begin{enumerate}[label=(\roman*)]
    \item \label{importantlemmag_1H^s_2g_2L^2}Suppose $w_1\in H^s_2(\R^3\times\R^3)$ and $w_2\in L^2(\R^3\times\R^3)$. Then for any $t\geq0$,
    \begin{align*}
        \norm{\int_{\R^3} \abs{\widehat{w_1}(\ell,\ell t)}\abs{\widehat{w_2}(k-\ell,\eta-\ell t)}\,d\ell}_{L^2_{k,\eta}}\lesssim_s\norm{w_1}_{H^s_2}\norm{w_2}_{L^2}.
    \end{align*}
    \item \label{importantlemmag_1H^0_2g_2H^s}Suppose $w_1\in H^0_2(\R^3\times\R^3)$ and $w_2\in H^s(\R^3\times\R^3)$. Then for any $t\geq0$,
    \begin{align*}
        \norm{\int_{\R^3} \abs{\widehat{w_1}(\ell,\ell t)}\abs{\widehat{w_2}(k-\ell,\eta-\ell t)}\,d\ell}_{L^2_{k,\eta}}\lesssim_s\norm{w_1}_{H^0_2}\norm{w_2}_{H^s}.
    \end{align*}
\end{enumerate}
\end{lemma}
\begin{proof}
    To prove~\ref{importantlemmag_1H^s_2g_2L^2}, Minkowski's integral inequality implies
    \begin{align*}
        \norm{\int_{\R^3} \abs{\widehat{w_1}(\ell,\ell t)}\abs{\widehat{w_2}(k-\ell,\eta-\ell t)}\,d\ell}_{L^2_{k,\eta}}\leq\norm{\widehat{w_2}}_{L^2}\int_{\R^3}\abs{\widehat{w_1}(\ell,\ell t)}\,d\ell.
    \end{align*}
The Sobolev embedding applied to the $\eta$ variable gives
\begin{align}
        \sup_{\eta}\abs{\widehat{w_1}(\ell,\eta)}\lesssim_s\sum_{\aal\leq 2}\br{\int_{\R^3}\abs{D^\alpha_\eta \widehat{w_1}(\ell,\eta)}^2\,d\eta}^{1/2}\label{Sobolevembedding}.
\end{align}
Using \eqref{Sobolevembedding} and Cauchy--Schwarz in $\ell$,
\begin{align*}
    \int_{\R^3}\abs{\widehat{w_1}(\ell,\ell t)}\,d\ell\lesssim_s\sum_{\aal\leq 2}\br{\int_{\R^3\times\R^3}\jb{\ell}^{2s}\abs{D^\alpha_\eta \widehat{w_1}(\ell,\eta)}^2\,d\eta d\ell}^{1/2}\br{\int_{\R^3}\jb{\ell}^{-2s}\,d\ell}^{1/2}\lesssim\norm{w_1}_{H^s_2}.
\end{align*}

To prove~\ref{importantlemmag_1H^0_2g_2H^s}, Minkowski's integral inequality applied in the $\eta$ variable, followed by Young's convolution inequality, implies
\begin{align*}
    \norm{\int_{\R^3} \abs{\widehat{w_1}(\ell,\ell t)}\abs{\widehat{w_2}(k-\ell,\eta-\ell t)}\,d\ell}_{L^2_{k,\eta}}&\leq \norm{\int_{\R^3} \abs{\widehat{w_1}(\ell,\ell t)}\norm{\widehat{w_2}(k-\ell,\eta)}_{L^2_\eta}\,d\ell}_{L^2_{k}}\\
    &\lesssim\norm{\widehat{w_1}(\ell,\ell t)}_{L^2_\ell}\norm{\widehat{w_2}(k,\eta)}_{L^1_kL^2_{\eta}}.
\end{align*}
Cauchy--Schwarz in $k$ then gives
\begin{align*}
    \norm{\widehat{w_2}(k,\eta)}_{L^1_kL^2_{\eta}}\lesssim\norm{\jb{k}^s\widehat{w_2}(k,\eta)}_{L^2_{k,\eta}}\lesssim_s\norm{w_2}_{H^s},
\end{align*}
and then Lemma~\ref{boundonenergy} concludes $\norm{\widehat{w_1}(\ell,\ell t)}_{L^2_\ell}\lesssim\norm{w_1}_{H^0_2}$.\end{proof}

\subsection{Sobolev norm estimates for linear and nonlinear operators}Sobolev estimates in $H^\sigma_2$ are developed for the linear and bilinear operators arising in the classical and quantum formulations.

As a first step, unweighted Sobolev estimates are derived.
\begin{lemma}[Unweighted $H^\sigma$ bounds]\label{unweightedlemma}
Let $\sigma\geq0$ and $s>3/2$. Assume $[K]_1<\infty$. 
\begin{enumerate}[label=(\roman*)]
    \item Suppose $\mu$ satisfies $\nabla\mu\in H^{\sigma}(\R^3)$ and $g(t)\in H^{\sigma}_2$. Then the linear operators $L$ and $L^\hbar$ satisfy, for any $t\geq0$,
    \begin{align*}
\norm{L_\mu[g](t)}_{H^\sigma}+\norm{L^\hbar_{\mu}[g](t)}_{H^\sigma}&\lesssim_{\sigma,K}\norm{\nabla\mu}_{H^\sigma}\norm{g(t)}_{H^\sigma_2}.
\end{align*}
\item Suppose $g_1(t)\in H^\sigma_2\cap H^s_2$ and $g_2(t)\in H^{s+1}\cap H^{\sigma+1}$. Then the nonlinear operators $N$ and $N^\hbar$ satisfy, for any $t\geq0$,
\begin{align*}
&\norm{N[g_1,g_2](t)}_{H^\sigma}+\norm{N^\hbar[g_1,g_2](t)}_{H^\sigma}\\
&\hspace{2cm}\lesssim_{\sigma,s,K}\jb{t}\br{\norm{g_1(t)}_{H^\sigma_2}\norm{g_2(t)}_{H^{s+1}}+\norm{g_1(t)}_{H^{s}_2}\norm{g_2(t)}_{H^{\sigma+1}}}
    \end{align*}
\end{enumerate}\end{lemma}
\begin{proof}
Arguments for the Vlasov operators on the Fourier side \(\mathcal{L}\) and \(\mathcal{N}\) given by \eqref{mathcalL} and \eqref{mathcalN} are provided. 
For the linear estimate, apply $\jb{k,\eta}^\sigma\lesssim\jb{k,kt}^\sigma\jb{\eta-kt}^\sigma$,
\begin{align*}
    \norm{\jb{k,\eta}^\sigma\mathcal L_\mu[g](t,k,\eta)}_{L^2_{k,\eta}}\lesssim_{K}\br{\int_{\R^3\times\R^3}\jb{k,kt}^{2\sigma}\abs{\widehat{g}(t,k,kt)}^2\jb{\eta-kt}^{2\sigma}\abs{\widehat{\nabla\mu}(\eta-kt)}^2\,d\eta dk}^{1/2}
\end{align*}where $[K]_1\lesssim1$ was used. Changing variables $\eta\mapsto\eta+kt$ and applying Lemma \ref{boundonenergy} yields the result. 

For the nonlinear estimate, use $\jb{k,\eta}^\sigma\lesssim\jb{\ell,\ell t}^\sigma+\jb{k-\ell,\eta-\ell t}^\sigma$ and $\abs{\eta-kt}\lesssim\jb{t}\jb{k-\ell,\eta-\ell t}$ (Lemma~\ref{japanesebracketlemma}~\ref{JBtriangle} and~\ref{JBshear}),
\begin{align*}
    &\norm{\jb{k,\eta}^\sigma\mathcal N[g_1,g_2](t,k,\eta)}_{L^2_{k,\eta}}\\
    &\hspace{0.5cm}\lesssim_K\jb{t}\norm{\int_{\R^3}\br{\jb{\ell,\ell t}^\sigma+\jb{k-\ell,\eta-\ell t}^\sigma}\abs{\widehat{g_1}(t,\ell,\ell t)}\jb{k-\ell,\eta-\ell t}\abs{\widehat{g_2}(t,k-\ell,\eta-\ell t)}\,d\ell}_{L^2_{k,\eta}}.
\end{align*} 
For the first contribution, apply Lemma~\ref{importantlemma}~\ref{importantlemmag_1H^0_2g_2H^s} with $w_1=\jb{\nabla_z,\nabla_\xi}^\sigma g_1(t)$ and $w_2=\jb{\nabla_z,\nabla_\xi}g_2(t)$,
\begin{align*}
    &\norm{\int_{\R^3}\jb{\ell,\ell t}^\sigma\abs{\widehat{g_1}(t,\ell,\ell t)}\jb{k-\ell,\eta-\ell t}\abs{\widehat{g_2}(t,k-\ell,\eta-\ell t)}\,d\ell}_{L^2_{k,\eta}}\lesssim\norm{g_1(t)}_{H^\sigma_2}\norm{g_2(t)}_{H^{s+1}}.
\end{align*}
For the second contribution, apply Lemma~\ref{importantlemma}~\ref{importantlemmag_1H^s_2g_2L^2} with $w_1=g_1(t)$ and $w_2=\jb{\nabla_z,\nabla_\xi}^{\sigma+1}g_2(t)$,
\begin{align*}
    \norm{\int_{\R^3}\jb{k-\ell,\eta-\ell t}^\sigma\abs{\widehat{g_1}(t,\ell,\ell t)}\jb{k-\ell,\eta-\ell t}\abs{\widehat{g_2}(t,k-\ell,\eta-\ell t)}\,d\ell}_{L^2_{k,\eta}}\lesssim\norm{g_1(t)}_{H^s_2}\norm{g_2(t)}_{H^{\sigma+1}}.
\end{align*}

The proof for the Hartree operators $\mathcal L^\hbar$ and $\mathcal N^\hbar$ given in \eqref{mathcalLhbar} and \eqref{mathcalNhbar} is identical upon applying \(\abs{\sin x}\le |x|\).
\end{proof}

The next lemma decomposes how the linear and nonlinear terms in the Vlasov and Hartree equations interact with the velocity weights $\xi^\alpha$.

\begin{lemma}[$\xi^\alpha$-weight decomposition]\label{decompositionlemma} Let $1\leq\aal\leq 2$, $\sigma\geq0$ and $s>3/2$. Assume $\mu\in H^\sigma_1(\R^3)$, $g(t)\in H^\sigma_2(\R^3\times\R^3)$, $g_1(t)\in H^{\max\set{\sigma,s}}_2(\R^3\times\R^3)$ and $g_2(t)\in H^{\max\set{\sigma,s}}_{\aal-1}(\R^3\times\R^3)$.
\begin{enumerate}[label=(\roman*)]
    \item Suppose $[K]_1<\infty$. Then the classical operators $L$ and $N$ satisfy
\begin{align*}
    \xi^\alpha L_\mu[g]&=L_{\xi^\alpha\mu}[g]+L_{<\alpha,\mu}[g],\\
    \xi^\alpha N[g_1,g_2]&= N[g_1,\xi^\alpha g_2]+N_{<\alpha}[g_1,g_2],
\end{align*}
where
\begin{align*}
    \norm{L_{<\alpha,\mu}[g](t)}_{H^\sigma}&\lesssim_{\sigma,K}\norm{g(t)}_{H^\sigma_2}\norm{\mu}_{H^\sigma_{\aal-1}},\\
    \norm{N_{<\alpha}[g_1,g_2](t)}_{H^\sigma}&\lesssim_{\sigma,s,K}\br{\norm{g_1(t)}_{H^\sigma_2}\norm{g_2(t)}_{H^{s}_{\aal-1}}+\norm{g_1(t)}_{H^s_2}\norm{g_2(t)}_{H^\sigma_{\aal-1}}}.
\end{align*}
\item Suppose $0<\hbar\leq 1$ and $[K]_2<\infty$. Then the quantum operators $L^\hbar$ and $N^\hbar$ satisfy\begin{align*}
    \xi^\alpha L^\hbar_\mu[g]&=L^\hbar_{\xi^\alpha\mu}[g]+L^\hbar_{<\alpha,\mu}[g],\\
    \xi^\alpha N^\hbar[g_1,g_2]&= N^\hbar[g_1,\xi^\alpha g_2]+N^\hbar_{<\alpha}[g_1,g_2],
\end{align*}
where
\begin{align*}
    \norm{L^\hbar_{<\alpha,\mu}[g](t)}_{H^\sigma}&\lesssim_{\sigma,K}\norm{g(t)}_{H^\sigma_2}\norm{\mu}_{H^\sigma_{\aal-1}},\\
    \norm{N^\hbar_{<\alpha}[g_1,g_2](t)}_{H^\sigma}&\lesssim_{\sigma,s,K}\br{\norm{g_1(t)}_{H^\sigma_2}\norm{g_2(t)}_{H^{s}_{\aal-1}}+\norm{g_1(t)}_{H^s_2}\norm{g_2(t)}_{H^\sigma_{\aal-1}}}.
\end{align*}
\end{enumerate}
\end{lemma}

\begin{proof}
Work on the Fourier side using that $\mathcal F_{z,\xi}\sbr{\xi^\alpha f(z,\xi)}(k,\eta)=(iD_\eta)^\alpha\widehat f(k,\eta)$.
   
   For the Vlasov equation, the Leibniz rule implies
\begin{align*}
   (iD_\eta)^\alpha\mathcal L_\mu[g](t,k,\eta)&=-i^\aal\widehat K(k)\widehat g(t,k,kt) k\cdot  \sum_{\beta\leq\alpha}c_{\alpha,\beta}D^\beta_\eta\sbr{\eta-kt}D^{\alpha-\beta}\widehat{\mu}(\eta-tk),\\
    (iD_\eta)^\alpha\mathcal N[g_1,g_2](t,k,\eta)&=-i^\aal(2\pi)^{-3}\int_{\R^3} \widehat K(\ell)\widehat {g_1}(t,\ell,\ell t)\ell\cdot \sum_{\beta\leq\alpha}c_{\alpha,\beta}D^\beta_\eta\sbr{\eta-kt}D^{\alpha-\beta}_\eta\widehat {g_2}(t,k-\ell,\eta-t\ell)\,d\ell.
\end{align*}The terms $\beta=0$ are exactly $\mathcal L_{\xi^\alpha\mu}[g](t,k,\eta)$ and $\mathcal N[g_1,\xi^\alpha g_2]$. For the remainders $L_{<\alpha,\mu}[g]$ and $N_{<\alpha}[g_1,g_2]$, apply
\begin{align*}
    \abs{D^\beta_{\eta}\sbr{\eta-kt}}\lesssim1.
\end{align*}This leads to, using $[K]_1\lesssim1$,
\begin{align}\label{VlasovremainderboundL}
        \abs{\mathcal L_{<\alpha,\mu}[g](t,k,\eta)}&\lesssim_K\abs{\widehat g(t,k,kt)}\sum_{\abs{\beta}<\aal}\abs{D^\beta\widehat\mu(\eta-kt)},\\
        \abs{\mathcal N_{<\alpha}[g_1,g_2](t,k,\eta)}&\lesssim_K\int_{\R^3} \abs{\widehat{g_1}(t,\ell,\ell t)}\sum_{\aab<\aal}\abs{D^\beta\widehat{g_2}(t,k-\ell ,\eta-\ell t)}\,d\ell.\label{VlasovremainderboundN}
    \end{align}
    For the linear term, applying $\jb{k,\eta}^\sigma\lesssim\jb{k,kt}^\sigma\jb{\eta-kt}^\sigma$ and Lemma~\ref{boundonenergy} implies
    \begin{align*}
        \norm{L_{<\alpha,\mu}[g](t)}_{H^\sigma}\lesssim\norm{\jb{k,kt}^\sigma \widehat g(t,k,kt)}_{L^2_k}\sum_{\abs{\beta}<\aal}\norm{\jb{\eta}^\sigma D^\beta\widehat{\mu}(\eta)}_{L^2_\eta}\lesssim\norm{g}_{H^\sigma_2}\norm{\mu}_{H^\sigma_{\aal-1}}.
    \end{align*}
    For the nonlinear term, decompose $\jb{k,\eta}^\sigma\lesssim\jb{\ell,\ell t}^\sigma+\jb{k-\ell,\eta-\ell t}^\sigma$. For the first contribution, apply Lemma~\ref{importantlemma}~\ref{importantlemmag_1H^0_2g_2H^s} with $w_1=\jb{\nabla_z,\nabla_\xi}^\sigma g_1$ and $w_2=\sum_{\aab<\aal}\xi^\beta g_2$,
    \begin{align*}
        \norm{\int_{\R^3} \jb{\ell,\ell t}^\sigma\abs{\widehat{g_1}(t,\ell,\ell t)}\sum_{\aab<\aal}\abs{D^\beta\widehat{g_2}(t,k-\ell ,\eta-\ell t)}\,d\ell}_{L^2_{k,\eta}}\lesssim\norm{g_1}_{H^\sigma_2}\norm{g_2}_{H^s_{\aal-1}}.
    \end{align*}
    For the second contribution, apply Lemma~\ref{importantlemma}~\ref{importantlemmag_1H^s_2g_2L^2} with $w_1=g_1$ and $w_2=\sum_{\aab<\aal}\jb{\nabla_z,\nabla_\xi}^\sigma(\xi^\beta g_2)$, yielding
    \begin{align*}
        \norm{\int_{\R^3} \abs{\widehat{g_1}(t,\ell,\ell t)}\sum_{\aab<\aal}\jb{k-\ell,\eta-\ell t}^\sigma\abs{D^\beta\widehat{g_2}(t,k-\ell ,\eta-\ell t)}\,d\ell}_{L^2_{k,\eta}}\lesssim\norm{g_1}_{H^s_2}\norm{g_2}_{H^\sigma_{\aal-1}}.
    \end{align*}
This concludes the proof for the Vlasov terms.

    For the Hartree equation, the Leibniz rule implies 
    \begin{align*}
        (iD_\eta)^\alpha\mathcal L^\hbar_\mu[g](t,k,\eta)&=-i^\aal2\hbar^{-1}\widehat K(k)\widehat g(t,k,kt)\sum_{\beta\leq\alpha}c_{\alpha,\beta}D_\eta^\beta\sbr{\sin\br{\mytextfrac{\hbar}{2}k\cdot(\eta-kt)}}D^{\alpha-\beta}\widehat \mu(\eta-kt),\\
         (iD_\eta)^\alpha\mathcal N^\hbar[g_1,g_2](t,k,\eta)&=-i^\aal2\hbar^{-1}(2\pi)^{-3}\int_{\R^3} \widehat{K}(\ell )\widehat{g_1}(t,\ell,\ell t)\\
         &\hspace{3cm}\times\sum_{\beta\leq\alpha}c_{\alpha,\beta}D_\eta^\beta\sbr{\sin\br{\mytextfrac{\hbar}{2}\ell \cdot (\eta-kt)}}D^{\alpha-\beta}_\eta\widehat{g_2}(t,k-\ell ,\eta-\ell t)\,d\ell.
    \end{align*}The terms $\beta=0$ are exactly $\mathcal L^\hbar_{\xi^\alpha\mu}[g](t,k,\eta)$ and $\mathcal N^\hbar[g_1,\xi^\alpha g_2]$. For the remainders $L_{<\alpha,\mu}^\hbar[g]$ and $N_{<\alpha}^\hbar[g_1,g_2]$, apply the chain rule and $\abs{\sin}$, $\abs{\cos}\leq1$,
    \begin{align*}
        2\hbar^{-1}\abs{D^\beta_\eta\sbr{\sin\br{\mytextfrac{\hbar}{2}x\cdot(\eta-kt)}}}\leq\hbar^{-1}(\hbar\abs{x})^{\aab}\leq\abs{x}\jb{\hbar x}
    \end{align*}for $1\leq\abs{\beta}\leq 2$.
    This leads to
    \begin{align*}
        \abs{\mathcal L_{<\alpha,\mu}^\hbar[g](t,k,\eta)}&\lesssim\jb{\hbar k}\abs{k}\abs{\widehat K(k)}\abs{\widehat g(t,k,kt)}\sum_{\abs{\beta}<\aal}\abs{D^\beta\widehat\mu(\eta-kt)},\\
        \abs{\mathcal N_{<\alpha}^\hbar[g_1,g_2](t,k,\eta)}&\lesssim\int_{\R^3} \jb{\hbar\ell}\abs{\ell}\abs{\widehat{K}(\ell)}\abs{\widehat{g_1}(t,\ell,\ell t)}\sum_{\aab<\aal}\abs{D^\beta\widehat{g_2}(t,k-\ell ,\eta-\ell t)}\,d\ell.
    \end{align*}
    The assumptions $0<\hbar\leq 1$ and $[K]_2\lesssim1$ allow us to control \begin{align*}
        \jb{\hbar k}\abs{k}\abs{\widehat K(k)}\lesssim1.
    \end{align*}The proof now concludes exactly as in the Vlasov case, with these bounds playing the roles of \eqref{VlasovremainderboundL} and \eqref{VlasovremainderboundN}.
\end{proof}

The following lemma combines Lemmas~\ref{unweightedlemma} and~\ref{decompositionlemma} to prove weighted $H^\sigma_2$ estimates on the linear and nonlinear terms.
\begin{lemma}[Weighted $H^\sigma_2$ bounds]\label{lemmaboundsonLN}
    Let $\sigma\geq0$ and $s>3/2$. Suppose $\mu\in H^{\sigma+1}_2(\R^3)$, $g(t)\in H^{\sigma}_2(\R^3\times\R^3)$, $g_1(t)\in H^{\max\set{\sigma,s}}_2(\R^3\times\R^3)$ and $g_2(t)\in H^{\max\set{\sigma,s}+1}_{2}(\R^3\times\R^3)$. 
    
    Suppose $[K]_1<\infty$. Then the classical operators $L$ and $N$ satisfy
    \begin{align*}
        \norm{L_\mu[g](t)}_{H^\sigma_2}&\lesssim_{\sigma,K}\norm{g(t)}_{H^\sigma_2}\norm{\mu}_{H^{\sigma+1}_2},\\
       \norm{ N[g_1,g_2](t)}_{H^\sigma_2}&\lesssim_{\sigma,s,K}\jb{t}\br{\norm{g_1(t)}_{H^\sigma_2}\norm{g_2(t)}_{H_2^{s+1}}+\norm{g_1(t)}_{H^{s}_2}\norm{g_2(t)}_{H_2^{\sigma+1}}}.
    \end{align*}

    Suppose $0<\hbar\leq 1$ and $[K]_2<\infty$. Then the quantum operators $L^\hbar$ and $N^\hbar$ satisfy
    \begin{align*}
         \norm{L^\hbar_{\mu}[g](t)}_{H^\sigma_2}&\lesssim_{\sigma,K}\norm{\mu}_{H^{\sigma+1}_2}\norm{g(t)}_{H^\sigma_2},\\
        \norm{N^\hbar [g_1,g_2](t)}_{H^\sigma_2}&\lesssim_{\sigma,s,K}\jb{t}\br{\norm{g_1(t)}_{H^\sigma_2}\norm{g_2(t)}_{H_2^{s+1}}+\norm{g_1(t)}_{H^{s}_2}\norm{g_2(t)}_{H_2^{\sigma+1}}}.
    \end{align*}
\end{lemma}
\begin{proof}
To estimate the weighted $H^\sigma_2$ norms, use Lemma \ref{decompositionlemma} to decompose into highest-order and remainder terms. The highest-order terms are controlled by Lemma \ref{unweightedlemma} applied to $\xi^\alpha\mu$ and $\xi^\alpha g_2$, and the remainder terms by the bounds in Lemma~\ref{decompositionlemma}.
\end{proof}

\subsection{Symmetric energy cancellation for nonlinear terms}A symmetric structure in the nonlinear terms yields improved energy estimates under minimal regularity assumptions.

Introduce the \(H^\sigma\) and weighted \(H^\sigma_M\) inner products by
\begin{align}\label{innerproductdefinition}
    \jb{g_1,g_2}_{H^\sigma}=(2\pi)^{-6}\int_{\R^3\times\R^3}\jb{k,\eta}^{2\sigma}\overline{\widehat{g_1}(k,\eta)}\widehat{g_2}(k,\eta)\,d\eta dk,\quad\jb{g_1,g_2}_{H^\sigma_2}=\sum_{\aal\leq 2}\jb{\xi^\alpha g_1,\xi^\alpha g_2}_{H^\sigma}.
\end{align}
The next lemma uses a symmetric cancellation in the bilinear structure to estimate\begin{align*}
    \abs{\Re\jb{g_2(t),N[g_1,g_2](t)\vphantom{2_2^2}}_{H^\sigma_2}}\qaq \abs{\Re\jb{g_2(t),N^\hbar[g_1,g_2](t)\vphantom{2_2^2}}_{H^\sigma_2}}.
\end{align*}Estimating these by Cauchy--Schwarz and Lemma~\ref{lemmaboundsonLN} would require $g_2(t)\in H^{\sigma+1}_2$, a loss of one derivative in the second argument; the cancellation requires only $g_2(t)\in H^\sigma_2$, which is what allows the energy estimates to close at the top regularity.

\begin{lemma}[Energy cancellation]Let $\sigma\geq1$ and $s>3/2$. Suppose $g_1(t),g_2(t)\in H^{\max\set{\sigma,s+1}}_2(\R^3\times\R^3)$ with $g_1(t)$ real-valued. 

Suppose $[K]_1<\infty$. Then the Vlasov nonlinear term satisfies
    \begin{align*}
        \abs{\Re\jb{g_2(t),N[g_1,g_2](t)\vphantom{2_2^2}}_{H^\sigma_2}}\lesssim_{\sigma,s,K}\jb{t}\norm{g_2(t)}_{H^\sigma_2}\br{\norm{g_1(t)}_{H^\sigma_2}\norm{g_2(t)}_{H^{s+1}_2}+\norm{g_1(t)}_{H^{s+1}_2}\norm{g_2(t)}_{H^\sigma_2}}.
    \end{align*}
    
Suppose $0<\hbar\leq 1$ and $[K]_2<\infty$. Then the Hartree nonlinear term satisfies 
\begin{align*}
        \abs{\Re\jb{g_2(t),N^\hbar[g_1,g_2](t)\vphantom{2_2^2}}_{H^\sigma_2}}\lesssim_{\sigma,s,K}\jb{t}\norm{g_2(t)}_{H^\sigma_2}\br{\norm{g_1(t)}_{H^\sigma_2}\norm{g_2(t)}_{H^{s+1}_2}+\norm{g_1(t)}_{H^{s+1}_2}\norm{g_2(t)}_{H^\sigma_2}}.
    \end{align*}\label{symmetrylemma}
    \end{lemma}
\begin{proof}The argument is first carried out for Schwartz functions \( g_1, g_2 \in \mathcal{S}(\R^3\times\R^3) \), for which all Fourier-space manipulations are justified; the final estimate extends to all \( g_1, g_2 \in H^{\max\{\sigma,s+1\}}_2(\R^3\times\R^3)\) with $g_1$ real-valued by density and continuity of the bilinear form. 

Consider first the case $\alpha=0$,
\begin{align*}
    \Re&\jb{ g_2(t), N[g_1,g_2](t)}_{H^\sigma}\\
    &=-(2\pi)^{-9}\Re\int_{\R^9}\jb{k,\eta}^\sigma\overline{\widehat{g_2}(t,k,\eta)}\widehat K(\ell)\widehat {g_1}(t,\ell,\ell t)\ell\cdot (\eta-kt)\jb{k,\eta}^\sigma\widehat {g_2}(t,k-\ell,\eta-t\ell)\,d\ell d\eta dk.
\end{align*}
Define
\begin{align*}
    I:=\int_{\R^9}\jb{k,\eta}^\sigma\overline{\widehat{g_2}(t,k,\eta)}\widehat K(\ell)\widehat {g_1}(t,\ell,\ell t)\ell\cdot (\eta-kt)\jb{k-\ell,\eta-\ell t}^\sigma\widehat {g_2}(t,k-\ell,\eta-t\ell)\,d\ell d\eta dk.
\end{align*}Make the change of variables $k\mapsto k+\ell$, $\eta\mapsto \eta+t\ell$, and $\ell\mapsto -\ell$, noting that $\eta-kt$ is invariant,
\begin{align*}
    I=-\int_{\R^9}\jb{k-\ell,\eta-\ell t}^\sigma\overline{\widehat{g_2}(t,k-\ell,\eta-\ell t)}\widehat K(-\ell)\widehat {g_1}(t,-\ell,-\ell t)\ell\cdot (\eta-kt)\jb{k,\eta}^\sigma\widehat {g_2}(t,k,\eta)\,d\ell d\eta dk=-\overline I,
\end{align*}
since $\widehat K(-\ell)=\overline{\widehat K(\ell)}$ and $\widehat {g_1}(t,-\ell,-\ell t)=\overline{\widehat{g_1}(t,\ell,\ell t)}$ because $K$ and $g_1$ are real-valued. Hence $\Re I=0$.

Hence subtracting this zero integral, 
\begin{align*}
    \Re\jb{ g_2(t), N[g_1,g_2](t)}_{H^\sigma}&=-(2\pi)^{-9}\Re\int_{\R^9}\jb{k,\eta}^\sigma\overline{\widehat{g_2}(t,k,\eta)}\widehat K(\ell)\widehat {g_1}(t,\ell,\ell t)\ell\cdot (\eta-kt)\\
    &\hspace{2.5cm}\times\br{\jb{k,\eta}^\sigma-\jb{k-\ell,\eta-\ell t}^\sigma}\widehat {g_2}(t,k-\ell,\eta-t\ell)\,d\ell d\eta dk.
\end{align*}
The reverse triangle inequality for Japanese brackets (Lemma~\ref{japanesebracketlemma}~\ref{JBreverse}) with $\sigma\geq1$ implies
\begin{align*}
    \abs{\jb{k,\eta}^\sigma-\jb{k-\ell,\eta-\ell t}^\sigma}\lesssim\abs{\ell,\ell t}\br{\jb{\ell,\ell t}^{\sigma-1}+\jb{k-\ell,\eta-t\ell}^{\sigma-1}}.
\end{align*}
Therefore, applying Cauchy--Schwarz in $k,\eta$, $\abs{\ell\widehat{K}(\ell)}\lesssim1$ and $\abs{\eta-kt}\lesssim\jb{t}\jb{k-\ell,\eta-\ell t}$ (which follows from Lemma~\ref{japanesebracketlemma}~\ref{JBshear}),
\begin{align*}
    &\abs{\Re\jb{ g_2(t), N[g_1,g_2](t)\vphantom{2_2^2}}_{H^\sigma}}\\
    &\lesssim_K \jb{t}\norm{g_2(t)}_{H^\sigma}\norm{\int_{\R^3}\br{\jb{\ell,\ell t}^{\sigma-1}+\jb{k-\ell,\eta-t\ell}^{\sigma-1}}\abs{\reallywidehat {\jb{\nabla}g_1}(t,\ell,\ell t)}\abs{\reallywidehat {\jb{\nabla}g_2}(t,k-\ell,\eta-t\ell)}\,d\ell}_{L^2_{k,\eta}}.
\end{align*}
By the triangle inequality in $L^2_{k,\eta}$, the right-hand side splits into two contributions according to the two terms of the bracket. Applying Lemma~\ref{importantlemma}~\ref{importantlemmag_1H^0_2g_2H^s} to the first contribution and Lemma~\ref{importantlemma}~\ref{importantlemmag_1H^s_2g_2L^2} to the second yields
\begin{align*}
    \abs{\Re\jb{ g_2(t), N[g_1,g_2](t)\vphantom{2_2^2}}_{H^\sigma}}\lesssim\jb{t}\norm{g_2(t)}_{H^\sigma}\br{\norm{g_1(t)}_{H^\sigma_2}\norm{g_2(t)}_{H^{s+1}}+\norm{g_1(t)}_{H^{s+1}_2}\norm{g_2(t)}_{H^\sigma}}.
\end{align*}

If $1\leq\aal\leq 2$ then use Lemma \ref{decompositionlemma} to decompose \begin{align*}
        \Re\jb{\xi^\alpha g_2(t),\xi^\alpha N[g_1,g_2](t)}_{H^\sigma}=\Re\jb{\xi^\alpha g_2(t),N[g_1,\xi^\alpha g_2](t)}_{H^\sigma}+\Re\jb{\xi^\alpha g_2(t),N_{<\alpha}[g_1, g_2](t)}_{H^\sigma}.
    \end{align*} For the highest-order term, the case $\alpha=0$ applied with $\xi^\alpha g_2$ in place of $g_2$ implies that
\begin{align*}
    \abs{\Re\jb{\xi^\alpha g_2(t),N[g_1,\xi^\alpha g_2](t)\vphantom{2_2^2}}_{H^\sigma}}\lesssim_K\jb{t}\norm{g_2(t)}_{H^\sigma_2}\br{\norm{g_1(t)}_{H^\sigma_2}\norm{g_2(t)}_{H^{s+1}_2}+\norm{g_1(t)}_{H^{s+1}_2}\norm{g_2(t)}_{H^\sigma_2}},
\end{align*}  for $\aal\leq 2$. For the remainder term, use
    \begin{align*}
        \abs{\Re\jb{\xi^\alpha g_2(t),N_{<\alpha}[g_1, g_2](t)\vphantom{2_2^2}}_{H^\sigma}}\leq\norm{g_2(t)}_{H^\sigma_2}\norm{N_{<\alpha}[g_1, g_2](t)}_{H^\sigma}
    \end{align*}and Lemma \ref{decompositionlemma} to conclude the result. 

    The proof for the Hartree nonlinear term follows by the same argument: the change of variables sends $\sin\br{\mytextfrac{\hbar}{2}\ell\cdot(\eta-kt)}$ to its negative, exactly as it does $\ell\cdot(\eta-kt)$, so the cancellation applies verbatim, and the additional requirement on $K$ arises from the application of Lemma~\ref{decompositionlemma}. In the applications below, $g_1=W^{\hbar}[P^\hbar]$ is real-valued because $P^\hbar$ is self-adjoint, whence $\overline{W^{\hbar}[P^\hbar](x,\xi)}=W^{\hbar}[P^\hbar](x,\xi)$.
\end{proof}

\subsection{Semiclassical error estimates}\label{quantumerrorsection}
The error operators arising from the difference between the classical and quantum dynamics, denoted $R^L$ and $R^N$ and defined in \eqref{defnsofquantum}, are analysed in weighted Sobolev norms, leading to $O(\hbar^2)$ bounds.

\begin{proposition}[Quantum error bounds]\label{PropositionRemainderTerms}
Let $\sigma\geq0$ and $s>3/2$. Assume $[K]_1<\infty$ and $\mu\in H^{\sigma+3}_2(\R^3)$. Suppose $g(t)\in H^{\sigma+2}_2(\R^3\times\R^3)$, $g_1\in H^{\max\set{\sigma,s}+2}_2(\R^3\times\R^3)$ and $g_2\in H^{\max\set{\sigma,s}+3}_2(\R^3\times\R^3)$. Then, for all $\hbar>0$, 
\begin{align}
        \label{boundonRL}\norm{R^{L}_\mu[g](t)}_{H^\sigma_2}&\lesssim_{\sigma,K}\hbar^2\jb{t}^{-2}\norm{\mu}_{H^{\sigma+3}_2}\norm{g(t)}_{H^{\sigma+2}_2},\\
        \norm{R^{N}[g_1,g_2](t)}_{H^\sigma_2}&\lesssim_{\sigma,s,K}\hbar^2\jb{t}\br{\norm{g_1(t)}_{H^{s+2}_2}\norm{g_2(t)}_{H^{\sigma+3}_2}+\norm{g_1(t)}_{H^{\sigma+2}_2}\norm{g_2(t)}_{H^{s+3}_2}}.\label{boundonRN}
\end{align}
\end{proposition}\begin{proof}For the linear remainder term, define $\mathcal R^L_\mu[g](t,k,\eta)=\widehat{R^L_\mu[g]}(t,k,\eta)$, which satisfies 
\begin{align*}
    \mathcal R^{L}_\mu[g](t,k,\eta)&=-\widehat K(k)\widehat {g}(t,k,kt)2\hbar^{-1}\phi\br{\frac{\hbar}{2}k\cdot(\eta-kt)}\widehat\mu(\eta-kt),\quad\phi(x)=x-\sin x.
\end{align*} 
For $\aal\leq 2$, 
\begin{align*}
        D^\alpha_\eta\mathcal R^{L}_\mu[g](t,k,\eta)=-\widehat K(k)\widehat {g}(t,k,kt)\sum_{\beta\leq \alpha}c_{\alpha,\beta}D^\beta_\eta\sbr{2\hbar^{-1}\phi\br{\frac{\hbar}{2}k\cdot(\eta-kt)}}D^{\alpha-\beta}\widehat\mu(\eta-kt).
    \end{align*}
By the elementary inequalities $\abs{\sin x}\leq \abs{x}$, $\abs{1-\cos x}\leq\tfrac{\abs{x}^2}{2}$ and $\abs{x-\sin x}\leq\tfrac{\abs{x}^3}{6}$, it follows that\begin{align*}
    \abs{\frac{d^j}{dx^j}\phi(x)}\leq \abs{x}^{3-j},\quad j=0,1,2.
\end{align*}
Hence, by the chain rule, for any multi-index $\beta$ with $\abs{\beta}\leq 2$,
\begin{align}
    2\hbar^{-1}\abs{D^\beta_\eta\sbr{\phi\br{\frac{\hbar}{2}k\cdot(\eta-kt)}}}\leq \hbar^2 \ak^3\abs{\eta-kt}^{3-\aab}.\label{phiinequalityforlinear}
\end{align}
Using \eqref{phiinequalityforlinear}, $[K]_1\lesssim1$ and $\abs{k}=\abs{k,kt}/\jb{t}\leq\jb{k,kt}/\jb{t}$ yields
\begin{align*}
    \abs{D^\alpha_\eta\mathcal R^{L}_\mu[g](t,k,\eta)}&\lesssim_K\hbar^2\jb{t}^{-2}\jb{k,kt}^{2}\abs{\widehat {g}(t,k,kt)}\jb{\eta-kt}^{3}\sum_{\beta\leq\alpha}\abs{D^\beta_\eta\widehat{\mu}(\eta-kt)}.
\end{align*}
Applying Lemma~\ref{japanesebracketlemma}~\ref{JBproduct} in the form $\jb{k,\eta}^\sigma\lesssim\jb{k,kt}^\sigma\jb{\eta-kt}^\sigma$, taking the $L^2_{k,\eta}$ norm, and applying Lemma~\ref{boundonenergy}, gives \begin{align*}
    \sum_{\aal\leq2}\norm{\jb{k,\eta}^\sigma D^\alpha_\eta\mathcal R^{L}_\mu[g](t,k,\eta)}_{L^2_{k,\eta}}&\lesssim\hbar^2\jb{t}^{-2}\norm{\mu}_{H^{\sigma+3}_2}\norm{g(t)}_{H^{\sigma+2}_2}.
\end{align*}Then \eqref{boundonRL} follows by Plancherel.

For the nonlinear term, define $\mathcal R^N[g_1,g_2](t,k,\eta)=\reallywidehat{R^N[g_1,g_2]}(t,k,\eta)$, which admits the representation
\begin{align*}
    \mathcal R^{N}[g_1,g_2](t,k,\eta)&=-(2\pi)^{-3}\int_{\R^3} \widehat{K}(\ell )\widehat {g_1}(t,\ell,\ell t)2\hbar^{-1}\phi\br{\tfrac{\hbar}{2}\ell\cdot(\eta-kt)}\widehat{g_2}(t,k-\ell ,\eta-\ell t)\,d\ell.
\end{align*}
As in the linear case, 
\begin{align*}
    2\hbar^{-1}\abs{D^\beta_\eta\sbr{\phi\br{\frac{\hbar}{2}\ell\cdot(\eta-kt)}}}\leq \hbar^2 \abs{\ell}^3\abs{\eta-kt}^{3-\aab},\quad\abs{\beta}\leq 2,
\end{align*}
hence for $\aal\leq 2$ the Leibniz rule and $\abs{\eta-kt}^{3-\aab}\leq\jb{\eta-kt}^3$ imply
\begin{align*}
    \abs{D^\alpha_\eta\mathcal R^{N}[g_1,g_2](t,k,\eta)}&\lesssim\hbar^2\sum_{\beta\leq\alpha}\int_{\R^3} \abs{\widehat{K}(\ell )}\abs{\widehat {g_1}(t,\ell,\ell t)}\abs{\ell}^3\jb{\eta-kt}^3\abs{D^{\alpha-\beta}_\eta\widehat{g_2}(t,k-\ell ,\eta-\ell t)}\,d\ell.
\end{align*}
Using $[K]_1\lesssim1$, $\jb{\eta-kt}\lesssim\jb{t}\jb{k-\ell,\eta-\ell t}$ (Lemma~\ref{japanesebracketlemma}~\ref{JBshear}), and $\abs{\ell}\leq\jb{\ell,\ell t}/\jb{t}$,
\begin{align*}
    \abs{D^\alpha_\eta\mathcal R^{N}[g_1,g_2](t,k,\eta)}&\lesssim_K\hbar^2\jb{t}\sum_{\beta\leq\alpha}\int_{\R^3} \jb{\ell,\ell t}^2\abs{\widehat {g_1}(t,\ell,\ell t)}\jb{k-\ell,\eta-\ell t}^3\abs{D^{\alpha-\beta}_\eta\widehat{g_2}(t,k-\ell ,\eta-\ell t)}\,d\ell.
\end{align*}
By Lemma~\ref{japanesebracketlemma}~\ref{JBtriangle}, in the form $\jb{k,\eta}^\sigma\lesssim\jb{\ell,\ell t}^\sigma+\jb{k-\ell,\eta-\ell t}^\sigma$, decompose
\begin{align*}
    \jb{k,\eta}^\sigma\abs{D^\alpha_\eta\mathcal R^{N}[g_1,g_2](t,k,\eta)}\lesssim R(t,k,\eta)+T(t,k,\eta),
\end{align*}
where
\begin{align*}
    R(t,k,\eta)&=\hbar^2\jb{t}\sum_{\beta\leq\alpha}\int_{\R^3} \jb{\ell,\ell t}^{\sigma+2}\abs{\widehat {g_1}(t,\ell,\ell t)}\jb{k-\ell,\eta-\ell t}^3\abs{D^{\alpha-\beta}_\eta\widehat{g_2}(t,k-\ell ,\eta-\ell t)}\,d\ell,\\
    T(t,k,\eta)&=\hbar^2\jb{t}\sum_{\beta\leq\alpha}\int_{\R^3} \jb{\ell,\ell t}^2\abs{\widehat {g_1}(t,\ell,\ell t)}\jb{k-\ell,\eta-\ell t}^{\sigma+3}\abs{D^{\alpha-\beta}_\eta\widehat{g_2}(t,k-\ell ,\eta-\ell t)}\,d\ell.
\end{align*}

For the $L^2_{k,\eta}$ norm of $T(t)$, apply Minkowski's inequality in $k,\eta$ and Cauchy--Schwarz in $\ell$, followed by Lemma~\ref{boundonenergy}:
\begin{align*}
    \norm{T(t,k,\eta)}_{L^2_{k,\eta}}&\lesssim\hbar^2\jb{t}\sum_{\beta\leq\alpha}\norm{\jb{\ell,\ell t}^2\jb{\ell}^s\widehat {g_1}(t,\ell,\ell t)}_{L^2}\norm{\jb{k,\eta}^{\sigma+3}D^\beta_\eta\widehat {g_2}(t,k,\eta)}_{L^2_{k,\eta}}\\
    &\lesssim\hbar^2\jb{t}\norm{g_1(t)}_{H^{s+2}_2}\norm{g_2(t)}_{H^{\sigma+3}_2},
\end{align*}for $s>3/2$.

Similarly, for $R(t)$, Minkowski's inequality in $\eta$, Young's convolution inequality in $k$ and Lemma~\ref{boundonenergy} give
\begin{align*}
    \norm{R(t,k,\eta)}_{L^2_{k,\eta}}&\lesssim\hbar^2\jb{t}\sum_{\beta\leq \alpha}\norm{\int_{\R^3}\jb{\ell,\ell t}^{\sigma+2}\abs{\widehat{g_1}(t,\ell,\ell t)}\norm{\jb{k-\ell,\cdot}^{3}D^\beta_\eta\widehat{g_2}(t,k-\ell,\cdot)}_{L^2_\eta}\,d\ell}_{L^2_k}\\
    &\lesssim\hbar^2\jb{t}\norm{\jb{\ell,\ell t}^{\sigma+2}\widehat{g_1}(t,\ell,\ell t)}_{L^2_\ell}\sum_{\beta\leq \alpha}\norm{\jb{k,\eta}^{3}D^\beta_\eta\widehat{g_2}(t,k,\eta)}_{L^1_kL^2_\eta}\\
    &\lesssim\hbar^2\jb{t}\norm{g_1(t)}_{H^{\sigma+2}_2}\norm{g_2(t)}_{H^{s+3}_2}.
\end{align*}
The bound \eqref{boundonRN} follows.
\end{proof}

\section{Finite-time semiclassical convergence}\label{Gronwallsection}
The energy and error estimates are now combined into a closed differential inequality, proving Theorem~\ref{maintheorem1}. The following form of Gr\"onwall's inequality, immediate from the standard differential form~\cite[Appendix~B.2]{Evans2010}, converts it into the integral bound of Theorem~\ref{SemiclassicalLimitTheorem}.

\begin{lemma}[Gr\"onwall inequality]\label{gronwalllemma}
    Let $T>0$, let $a,b\in L^1(0,T)$ be non-negative, and let $X\colon[0,T]\to[0,\infty)$ be continuous with $X^2$ absolutely continuous. Suppose that
    \begin{align*}
        \frac{1}{2}\frac{d}{dt}X(t)^2\leq a(t)X(t)^2+b(t)X(t)
    \end{align*}
    for almost every $t\in[0,T]$. Then, with $A(t):=\int_0^t a(s)\,ds$,
    \begin{align*}
        X(t)\leq X(0)e^{A(t)}+\int_0^t b(s)e^{A(t)-A(s)}\,ds.
    \end{align*}
\end{lemma}

The following theorem provides a quantitative estimate on the difference between the Hartree and Vlasov evolutions under two regularity regimes.

\begin{theorem}[Quantitative semiclassical limit]\label{SemiclassicalLimitTheorem}
    Let $\sigma>5/2$ and $T>0$. Let $g$ be a real-valued solution of~\eqref{equationforg(t)} on $[0,T]$ with $g(0)=h_{\rm in}$ and, for $\hbar>0$, let $P^\hbar(t)$ be a self-adjoint family with $g^\hbar=W^{\hbar}[P^\hbar]$ solving~\eqref{eqforg^hbar(t)} on $[0,T]$ and $P^\hbar(0)=Q_{\rm in}^\hbar$. Then the estimate
    \begin{align}\label{quantconclusion}
    \norm{g(t)-W^{\hbar}[P^\hbar(t)]}_{H^\sigma_2}\leq \norm{h_{\rm in}-W^{\hbar}[Q_{\mathrm{in}}^\hbar]}_{H^\sigma_2}e^{\mathcal A(t)}+\hbar^2\int_0^t \mathcal B(s) e^{\mathcal A(t)-\mathcal A(s)}\,ds,
    \qquad t\in[0,T],
    \end{align}
    holds in each of the following two regimes, with a constant $c_{\sigma,K}>0$ depending only on $\sigma$ and the finite kernel seminorm assumed in the relevant regime.
    \begin{enumerate}[label=(\alph*)]
        \item\label{aSLT} (Higher regularity on the Hartree solution) If $[K]_1<\infty$, then~\eqref{quantconclusion} holds for every $\hbar>0$ with
\begin{align}
        \begin{split}
            \mathcal{A}(t)&:=c_{\sigma,K}\int_0^t\br{\norm{\mu}_{H^{\sigma+1}_2}+\jb{s}\norm{g(s)}_{H^\sigma_2}+\jb{s}\norm{P^\hbar(s)}_{\mathcal H^{\sigma+1}_2}}\,ds,\\
        \mathcal{B}(t)&:=c_{\sigma,K}\norm{P^\hbar(t)}_{\mathcal H^{\sigma+2}_2}\br{\jb{t}^{-2}\norm{\mu}_{H^{\sigma+3}_2}+\jb{t}\norm{P^\hbar(t)}_{\mathcal H^{\sigma+3}_2}}.\label{undera}
        \end{split}
\end{align}
        \item\label{bSLT} (Higher regularity on the Vlasov solution) If $[K]_2<\infty$, then~\eqref{quantconclusion} holds for every $\hbar\in(0,1]$ with
    \begin{align}\begin{split}
        \mathcal{A}(t)&:=c_{\sigma,K}\int_0^t\br{\norm{\mu}_{H^{\sigma+1}_2}+\jb{s}\norm{P^\hbar(s)}_{\mathcal H^\sigma_2}+\jb{s}\norm{g(s)}_{H^{\sigma+1}_2}}\,ds,\\
        \mathcal{B}(t)&:=c_{\sigma,K}\norm{g(t)}_{H^{\sigma+2}_2}\br{\jb{t}^{-2}\norm{\mu}_{H^{\sigma+3}_2}+\jb{t}\norm{g(t)}_{H^{\sigma+3}_2}}.\label{underb}\end{split}
\end{align}
    \end{enumerate}
\end{theorem}
\begin{proof}
  Taking the derivative of the difference for the energy estimate: \begin{align*}
      \frac{1}{2}\frac{d}{dt}\norm{g(t)-g^\hbar(t)}_{H^\sigma_2}^2&=\Re\jb{ g(t)-g^\hbar(t),\p_t\br{g(t)-g^\hbar(t)}}_{H^\sigma_2},
  \end{align*} where the $H^\sigma_2$-inner product is defined in~\eqref{innerproductdefinition}. To deduce the bounds~\eqref{undera}, use the Vlasov-frame decomposition in~\eqref{Twoframes}, yielding
\begin{align*}
    \frac{1}{2}&\frac{d}{dt}\norm{g(t)-g^\hbar(t)}_{H^\sigma_2}^2\\
    &=\Re\jb{ g(t)-g^\hbar(t),L_\mu[g-g^\hbar](t)+ N[g,g-g^\hbar](t)+ N[g-g^\hbar,g^\hbar](t)+ R^{L}_\mu[g^\hbar](t)+ R^{N}[g^\hbar,g^\hbar](t)}_{H^\sigma_2}.
\end{align*} 
For the contribution from $N[g,g-g^\hbar]$, apply the symmetry lemma, Lemma~\ref{symmetrylemma}, with $\sigma>5/2$,
\begin{align*}
    \abs{\Re\jb{g(t)-g^\hbar(t),N[g,g-g^\hbar](t)}_{H^\sigma_2}}\lesssim_K\jb{t}\norm{g(t)}_{H^{\sigma}_2}\norm{g(t)-g^\hbar(t)}_{H^{\sigma}_2}^2.
\end{align*} 
For all other terms use $\abs{\jb{f,g}_{H^\sigma_2}}\leq\norm{f}_{H^\sigma_2}\norm{g}_{H^\sigma_2}$. By Lemma~\ref{lemmaboundsonLN},
\begin{align*}
    \abs{\jb{g(t)-g^\hbar(t),L_\mu[g-g^\hbar](t)}_{H^\sigma_2}}&\lesssim \norm{\mu}_{H^{\sigma+1}_2}\norm{g(t)-g^\hbar(t)}_{H^\sigma_2}^2,\\
    \abs{\jb{g(t)-g^\hbar(t),N[g-g^\hbar,g^\hbar](t)}_{H^\sigma_2}}&\lesssim_K \jb{t}\norm{g^\hbar(t)}_{H^{\sigma+1}_2}\norm{g(t)-g^\hbar(t)}_{H^\sigma_2}^2.
\end{align*}
By Proposition~\ref{PropositionRemainderTerms},
\begin{align*}
    \abs{\jb{g(t)-g^\hbar(t),R^{L}_\mu[g^\hbar](t)}_{H^\sigma_2}}&\lesssim\hbar^2\jb{t}^{-2}\norm{\mu}_{H^{\sigma+3}_2}\norm{g^\hbar(t)}_{H^{\sigma+2}_2}\norm{g(t)-g^\hbar(t)}_{H^\sigma_2},\\
    \abs{\jb{g(t)-g^\hbar(t),R^{N}[g^\hbar,g^\hbar](t)}_{H^\sigma_2}}&\lesssim\hbar^2 \jb{t}\norm{g^\hbar(t)}_{H^{\sigma+3}_2}\norm{g^\hbar(t)}_{H^{\sigma+2}_2}\norm{g(t)-g^\hbar(t)}_{H^\sigma_2}.
\end{align*}

Altogether,
\begin{align*}
    \frac{1}{2}\frac{d}{dt}\norm{g(t)-g^\hbar(t)}_{H^\sigma_2}^2&\lesssim\br{\norm{\mu}_{H^{\sigma+1}_2}+\jb{t}\norm{g(t)}_{H^\sigma_2}+\jb{t}\norm{g^\hbar(t)}_{H^{\sigma+1}_2}}\norm{g(t)-g^\hbar(t)}_{H^\sigma_2}^2\\
    &\hspace{0.2cm}+\hbar^2\norm{g^\hbar(t)}_{H^{\sigma+2}_2}\br{\jb{t}^{-2}\norm{\mu}_{H^{\sigma+3}_2}+\jb{t}\norm{g^\hbar(t)}_{H^{\sigma+3}_2}}\norm{g(t)-g^\hbar(t)}_{H^\sigma_2}.
\end{align*}
Gr\"onwall's inequality, Lemma~\ref{gronwalllemma}, applied with $X(t):=\norm{g(t)-g^\hbar(t)}_{H^\sigma_2}$, yields~\eqref{quantconclusion} with~\eqref{undera}, after choosing $c_{\sigma,K}$ sufficiently large to absorb the implicit constants above; the norm identity of Definition~\ref{quantumnormsdefinition} converts the norms of $g^\hbar$ into those of $P^\hbar$. The bounds~\eqref{underb} follow similarly using the Hartree-frame decomposition in~\eqref{Twoframes}.
\end{proof}

The constants obtained in Theorem~\ref{SemiclassicalLimitTheorem} can now be specified in terms of the data and time horizon, yielding the full statement of Theorem~\ref{maintheorem1}.
\begin{proof}[Proof of Theorem~\ref{maintheorem1}]
Assume the conditions of Theorem~\ref{maintheorem1}~\ref{amaintheorem1} are satisfied. By Theorem~\ref{SemiclassicalLimitTheorem}~\ref{aSLT} and the bound $\int_0^T\jb{s}\,ds\leq T\jb{T}$,
\begin{align*}
    \sup_{t\in[0,T]}\mathcal A(t)\leq a_{\sigma,K}\br{T\norm{\mu}_{H^{\sigma+1}_2}+T\jb{T}M_T}
    \qaq
    \int_0^T\mathcal B(s)\,ds\leq b_{\sigma,K}M_T\br{\norm{\mu}_{H^{\sigma+3}_2}+T\jb{T}M_T},
\end{align*}
after enlarging $a_{\sigma,K},b_{\sigma,K}$ if necessary. Inserting these bounds into~\eqref{quantconclusion} yields~\eqref{mainmainconc} with $C_1(T),C_2(T)$ as in~\eqref{Cconstants}, for all $t\in[0,T]$ and $\hbar\in(0,\delta]$. The proof of Theorem~\ref{maintheorem1}~\ref{bmaintheorem1} follows similarly, after further enlarging the constants.
\end{proof}

\section{Scattering for the Vlasov equation via the semiclassical limit}\label{scatteringproofsection}
In this section, the finite-time estimate of Theorem~\ref{maintheorem1} is combined with the uniform-in-$\hbar$ scattering theory for the Hartree equation~\cite{smith2025phasemixinghartreeequation} to derive the classical large-time theory. The main result is Proposition~\ref{Vlasovscattering}: near Penrose-stable homogeneous states, for small data the Vlasov solution exists globally and scatters, with every bound inherited from the quantum dynamics through the semiclassical limit and no input from the classical damping theory.

\subsection{Uniform-in-$\hbar$ phase mixing for the Hartree equation}
The following theorem records, in the notation of the present paper, the two estimates required from~\cite{smith2025phasemixinghartreeequation}.

\begin{theorem}[Uniform-in-$\hbar$ Hartree phase mixing~{\cite{smith2025phasemixinghartreeequation}}]\label{Smithresult}
Let $\sigma>5/2$ and assume $K\in L^1(\R^3)$ with $[K]_2<\infty$. There exist exponents $\overline\sigma>\sigma_0>\sigma'>\sigma$, with $\sigma'\geq\sigma+3$ and $\overline\sigma\geq\sigma_0+1$, such that the following holds. Let $\mu\in H^{\overline\sigma}_4(\R^3)$ be non-negative and suppose that the pair $(K,\mu)$ satisfies the uniform Penrose condition~\eqref{Penrose2} with constants $\kappa>0$ and $\delta\in(0,1]$. Then there exists $\ep_0=\ep_0(\sigma,\mu,K,\kappa)>0$ such that, if the self-adjoint initial perturbations $\Qin^\hbar$ satisfy
\begin{align}\label{quantumsmallness}
\sup_{\hbar\in(0,\delta]}\sum_{\aal\le2}\norm{\bm{x}^\alpha \Qin^\hbar}_{\mathcal H^{\sigma_0}_2}\le\ep_0,
\end{align}
then, for every $\hbar\in(0,\delta]$, the solution $Q^\hbar$ of~\eqref{NH} exists globally, and there exist constants $\alpha_0,C_b>0$, depending only on $\sigma$, $\mu$, $K$ and $\kappa$, together with scattering profiles $Q^\hbar_\infty\in\mathcal H^\sigma_2$, such that, for all $t\geq0$,
\begin{align}
\sup_{\hbar\in(0,\delta]}\norm{P^\hbar(t)-Q^\hbar_\infty}_{\mathcal H^\sigma_2}&\leq\frac{\alpha_0}{\jb{t}^{3/2}},\label{qhung}\\
\sup_{\hbar\in(0,\delta]}\norm{P^\hbar(t)}_{\mathcal H^{\sigma'}_2}&\leq C_b\jb{t}^{5/2}.\label{qhung2}
\end{align}
\end{theorem}
\begin{proof}
Apply \cite[Theorem~1.6]{smith2025phasemixinghartreeequation} in dimension $d=3$ with $w=K$ and $g=\mu$, so that the steady states and Penrose response functions of the two papers coincide, and identify its scattering, bootstrap and initial-data indices with the present $\sigma$, $\sigma'$ and $\sigma_0$. The conditions on the exponents prescribe only lower bounds on their successive gaps, so these choices are admissible with $\sigma'\geq\sigma+3$, and $\overline\sigma$ may be taken large enough for the regularity of $\mu$; the kernel decay required there follows from $[K]_2<\infty$. The uniform Penrose condition of~\cite{smith2025phasemixinghartreeequation} is formulated for all $\hbar\in(0,1]$, whereas~\eqref{Penrose2} is assumed only for $\hbar\in(0,\delta]$. The argument there is carried out at each fixed $\hbar$ and uses the Penrose bound only at that value of $\hbar$, with constants depending on the Penrose constant but not on $\hbar$, so its conclusions hold for every $\hbar\in(0,\delta]$ under~\eqref{Penrose2}. After decreasing $\ep_0$ if necessary, hypothesis~\eqref{quantumsmallness} implies the smallness condition on the initial data. Its conclusions are the global existence and scattering statements together with~\eqref{qhung}. Finally, the global bootstrap control $\mathrm{(B1)}$ of \cite[Proposition~1.22]{smith2025phasemixinghartreeequation} controls the $\mathcal H^{\sigma'}_2$ norm by $C\jb{t}^{5/2}$ in dimension three, which is~\eqref{qhung2}.
\end{proof}

\subsection{The Penrose conditions}
To transfer Theorem~\ref{Smithresult} to the Vlasov equation, the classical Penrose condition appearing, for example, in~\cite{mou-vil-2011,BedrossianMasmoudiMouhot2018} is now recalled and compared with the uniform condition~\eqref{Penrose2}.

\begin{definition}[Classical Penrose condition]\label{classicalPenrosedefinition}
Let $K\colon\R^3\to\R$ be an interaction kernel and $\mu\in L^1(\R^3;\R_+)$. The pair $(K,\mu)$ satisfies the \textit{classical Penrose condition} if there exists $\kappa>0$ such that
\begin{align}\label{Penrose1}
    \inf_{\substack{\Re\lambda \geq 0 \\ k \in \R^3}} \abs{1 + \widehat{K}(k) \int_0^\infty e^{-\lam t}t \ak^2 \widehat{\mu}\br{kt}\,dt} \geq \kappa.
\end{align}
\end{definition}

The two conditions are equivalent up to constants: the uniform condition implies the classical one with the same constant, and conversely the classical condition implies the uniform one for $\hbar$ small, with half the constant, under slightly more decay of $\widehat\mu$.

\begin{lemma}[Equivalence of the Penrose conditions]\label{Penroseequivalence}
Let $K\in L^1(\R^3)$ and $\mu\in H^{r}_2(\R^3)$.
\begin{enumerate}[label=(\roman*)]
    \item\label{Penroseequivalencei} If $r>2$ and the pair $(K,\mu)$ satisfies the uniform Penrose condition~\eqref{Penrose2} with constant $\kappa>0$, then it satisfies the classical Penrose condition~\eqref{Penrose1} with the same constant $\kappa$.
    \item\label{Penroseequivalenceii} If $r>4$ and $[K]_2<\infty$, and the pair $(K,\mu)$ satisfies the classical Penrose condition~\eqref{Penrose1} with constant $\kappa>0$, then there exists $\delta_0\in(0,1]$, depending only on $r$, $\kappa$, $[K]_2$ and $\norm{\mu}_{H^r_2}$, such that the uniform Penrose condition~\eqref{Penrose2} holds with constants $\kappa/2$ and $\delta_0$.
\end{enumerate}
\end{lemma}
\begin{proof}
For $\Re\lambda\geq0$ and $k\in\R^3$, define the classical and quantum response functions appearing in~\eqref{Penrose1} and~\eqref{Penrose2}, respectively, by
\begin{align*}
\mathcal D_0(\lambda,k):=1+\widehat K(k)\int_0^\infty e^{-\lambda t}t\ak^2\widehat\mu(kt)\,dt,
\quad
\mathcal D_\hbar(\lambda,k):=1+\frac{2}{\hbar}\widehat K(k)\int_0^\infty e^{-\lambda t}\sin\br{\frac{1}{2}\hbar t\ak^2}\widehat\mu(kt)\,dt.
\end{align*}
The Sobolev embedding $H^2(\R^3)\hookrightarrow L^\infty(\R^3)$, applied to $\jb{\eta}^{r}\widehat\mu$, gives the decay
\begin{align}\label{mudecay}
\abs{\widehat\mu(\eta)}\lesssim_r\jb{\eta}^{-r}\norm{\mu}_{H^{r}_2},
\end{align}
and, together with the pointwise bound $\abs{\widehat K}\leq\norm{K}_{L^1}$, this makes both response functions well defined, with $\mathcal D_\hbar(\lambda,0)=\mathcal D_0(\lambda,0)=1$.

\ref{Penroseequivalencei} Fix $k\neq0$ and $\Re\lambda\geq0$. As $\hbar\to0$,
\begin{align*}
    \frac{2}{\hbar}\sin\br{\frac{1}{2}\hbar t\ak^2}\rightarrow t\ak^2
    \qaq
    \abs{e^{-\lam t}\frac{2}{\hbar}\sin\br{\frac{1}{2}\hbar t\ak^2}\widehat\mu(kt)}\leq t\ak^2\abs{\widehat\mu(kt)},
\end{align*}
and the dominating function is integrable: by~\eqref{mudecay} and the substitution $s=\ak t$,
\begin{align*}
    \int_0^\infty t\ak^2\abs{\widehat\mu(kt)}\,dt=\int_0^\infty s\abs{\widehat\mu(\tfrac{k}{\ak}s)}\,ds\lesssim_r\norm{\mu}_{H^{r}_2}\int_0^\infty s\jb{s}^{-r}\,ds<\infty,
\end{align*}
as $r>2$. By dominated convergence, $\mathcal D_\hbar(\lambda,k)\to\mathcal D_0(\lambda,k)$ as $\hbar\to0$. Consequently, if $\abs{\mathcal D_\hbar}\geq\kappa$ for all $\hbar\in(0,\delta]$, then $\abs{\mathcal D_0}\geq\kappa$, which is~\eqref{Penrose1} with the same constant.

\ref{Penroseequivalenceii} The elementary inequality $\abs{\sin z-z}\leq\abs{z}^3/6$ and $\abs{e^{-\lambda t}}\leq1$ for $\Re\lambda\geq0$ give
\begin{align}\label{responsefunctiondifference}
\abs{\mathcal D_\hbar(\lambda,k)-\mathcal D_0(\lambda,k)}
\leq
\frac{\hbar^2}{24}\abs{\widehat K(k)}\ak^6\int_0^\infty t^3\abs{\widehat\mu(kt)}\,dt.
\end{align}
For $k\neq0$, the substitution $s=\ak t$ and the decay~\eqref{mudecay} yield
\begin{align*}
\ak^6\int_0^\infty t^3\abs{\widehat\mu(kt)}\,dt
=\ak^2\int_0^\infty s^3\abs{\widehat\mu(\tfrac{k}{\ak}s)}\,ds
\lesssim_r\ak^2\norm{\mu}_{H^r_2}\int_0^\infty s^3\jb{s}^{-r}\,ds
\lesssim_r\ak^2\norm{\mu}_{H^r_2},
\end{align*}
the final integral being finite because $r>4$. Since $\ak^2\abs{\widehat K(k)}\leq[K]_2$, and both response functions equal $1$ at $k=0$, estimate~\eqref{responsefunctiondifference} gives
\begin{align}\label{uniformresponseconvergence}
\sup_{\substack{\Re\lambda\geq0\\k\in\R^3}}
\abs{\mathcal D_\hbar(\lambda,k)-\mathcal D_0(\lambda,k)}
\leq
C_r\hbar^2[K]_2\norm{\mu}_{H^r_2}.
\end{align}
Choose $\delta_0\in(0,1]$ so small that $C_r\delta_0^2[K]_2\norm{\mu}_{H^r_2}\leq\kappa/2$. Then, for every $\hbar\in(0,\delta_0]$, $\Re\lambda\geq0$ and $k\in\R^3$,
\begin{align*}
\abs{\mathcal D_\hbar(\lambda,k)}
\geq
\abs{\mathcal D_0(\lambda,k)}-\abs{\mathcal D_\hbar(\lambda,k)-\mathcal D_0(\lambda,k)}
\geq
\kappa-\frac{\kappa}{2}=\frac{\kappa}{2},
\end{align*}
by~\eqref{Penrose1} and~\eqref{uniformresponseconvergence}, which is~\eqref{Penrose2} with constants $\kappa/2$ and $\delta_0$.
\end{proof}

\subsection{Transfer to the Vlasov equation}
Global existence and scattering for the Vlasov equation are now deduced from the quantum estimates, with no input from the classical damping literature. Given a small classical datum, quantise it exactly for each~$\hbar$ by taking its inverse Wigner transform. The resulting quantum data are exactly aligned with the classical datum, and the corresponding Hartree solutions scatter uniformly in~$\hbar$ by Theorem~\ref{Smithresult}. Combining these quantum estimates with the finite-time semiclassical estimate of Theorem~\ref{maintheorem1} and the standard local well-posedness and propagation theory for the Vlasov equation recorded in Appendix~\ref{PropagationAppendixReference} transfers the quantum bounds to the Vlasov solution in the limit $\hbar\to0$.

The resulting proposition derives, from quantum scattering and the semiclassical limit, a Sobolev-space scattering result for the Vlasov equation of the type established by Bedrossian, Masmoudi and Mouhot~\cite{BedrossianMasmoudiMouhot2018}, thereby completing the programme described in~\cite[Section~1.7]{smith2025phasemixinghartreeequation}. Since the hypotheses are inherited from the quantum theory, they are not optimised from the classical perspective; the sharper classical thresholds are those of~\cite{BedrossianMasmoudiMouhot2018}.

\begin{proposition}[Vlasov scattering via the semiclassical limit]\label{Vlasovscattering}
Let $\sigma>5/2$. Assume $K\in L^1(\R^3)$ with $[K]_2<\infty$, and let $\overline\sigma>\sigma_0>\sigma'>\sigma$ be the exponents of Theorem~\ref{Smithresult}. Let $\mu\in H^{\overline\sigma}_4(\R^3)$ be non-negative, and suppose that the pair $(K,\mu)$ satisfies the classical Penrose condition~\eqref{Penrose1} with constant $\kappa>0$. Then there exists $\ep_0=\ep_0(\sigma,\mu,K,\kappa)>0$ such that, if $\hin\in H^{\sigma_0}_2(\R^3\times\R^3)$ is real-valued and satisfies
\begin{align}\label{classicalsmallness}
\sum_{\aal\le2}\norm{x^\alpha\hin}_{H^{\sigma_0}_2}\le\ep_0,
\end{align}
then the solution $g$ of the Vlasov equation~\eqref{equationforg(t)} with initial datum $\hin$ exists globally, $g \in C([0,\infty); H^{\sigma_0}_2)$, the scattering profile
\begin{equation*}
	h_\infty := \lim_{t \to \infty} g(t) \quad \text{in } H^{\sigma}_2
\end{equation*}
exists, and there exist constants $\alpha_0, C_b > 0$, depending only on $\sigma$, $\mu$, $K$ and $\kappa$, and in particular independent of $\hin$, such that, for all $t \ge 0$,
\begin{align}
	\norm{g(t) - h_\infty}_{H^\sigma_2} &\le \frac{\alpha_0}{\jb{t}^{3/2}}, \label{classicalscatteringtransfer} \\
	\norm{g(t)}_{H^{\sigma'}_2} &\le C_b \jb{t}^{5/2}. \label{classicalregularitytransfer}
\end{align}
\end{proposition}
\begin{proof}
Since $\overline\sigma>\sigma_0>4$ and $\mu\in H^{\overline\sigma}_4\subset H^{\overline\sigma}_2$, Lemma~\ref{Penroseequivalence}~\ref{Penroseequivalenceii} provides $\delta_0\in(0,1]$ such that the pair $(K,\mu)$ satisfies the uniform Penrose condition~\eqref{Penrose2} with constants $\kappa/2$ and $\delta_0$. Define $\ep_0$ as the smallness threshold provided by Theorem~\ref{Smithresult} for this pair of constants.

\emph{Step 1: exactly aligned quantum data.}
Let $\hbar\in(0,\delta_0]$. For $f\in L^2(\R^3\times\R^3)$, the \emph{inverse Wigner transform} of $f$, equivalently $(2\pi\hbar)^3$ times its Weyl quantisation, is the operator $(W^\hbar)^{-1}[f]$ on $L^2(\R^3)$ with integral kernel
\begin{align*}
(W^\hbar)^{-1}[f](x,y):=\int_{\R^3}e^{i\xi\cdot(x-y)/\hbar}f\br{\frac{x+y}{2},\xi}\,d\xi,
\end{align*}
which satisfies $W^{\hbar}[(W^\hbar)^{-1}[f]]=f$ by Definition~\ref{Wignerdefinition} and Fourier inversion. Define
\begin{align*}
\Qin^\hbar:=(W^\hbar)^{-1}[\hin].
\end{align*}
Since $W^{\hbar}[\gamma^*]=\overline{W^{\hbar}[\gamma]}$ by Definition~\ref{Wignerdefinition}, the operator $\Qin^\hbar$ is self-adjoint because $\hin$ is real-valued. By Lemma~\ref{operatorsWigner}~\ref{intertwiningproperty} and Definition~\ref{quantumnormsdefinition},
\begin{align*}
\sum_{\aal\le2}\norm{\bm{x}^\alpha\Qin^\hbar}_{\mathcal H^{\sigma_0}_2}
=\sum_{\aal\leq2}\norm{x^\alpha\hin}_{H^{\sigma_0}_2}\le\ep_0,
\end{align*}
so~\eqref{quantumsmallness} holds with $\delta:=\delta_0$. Theorem~\ref{Smithresult} therefore applies: the Hartree solutions $Q^\hbar$ exist globally and satisfy~\eqref{qhung} and~\eqref{qhung2}. 

\emph{Step 2: fixed-time semiclassical convergence on the lifespan.}
Since $\overline\sigma\geq\sigma_0+1$, Proposition~\ref{prop:vlasovLWP}, applied at regularity $\sigma_0$, gives $T_{\max}\in(0,\infty]$ and a unique maximal solution $g\in C([0,T_{\max});H^{\sigma_0}_2)$. Fix $T\in(0,T_{\max})$. Since $\sigma_0>\sigma$, continuity gives $m_T:=\norm{g}_{C([0,T];H^\sigma_2)}<\infty$. Moreover,~\eqref{qhung2} and $\sigma'\geq\sigma+3$ imply
\begin{align*}
\sup_{\hbar\in(0,\delta_0]}
\norm{P^\hbar}_{C([0,T];\mathcal H^{\sigma+3}_2)}
\leq C_b\jb{T}^{5/2}.
\end{align*}
Thus hypothesis~\ref{amaintheorem1} of Theorem~\ref{maintheorem1} holds on $[0,T]$ with $M_T:=m_T+C_b\jb{T}^{5/2}$. Since the data are exactly aligned, the first term in~\eqref{mainmainconc} vanishes, and
\begin{align}\label{transferfixedtime}
\sup_{t\in[0,T]}
\norm{g(t)-W^\hbar[P^\hbar(t)]}_{H^\sigma_2}
\leq C_2(T)\hbar^2.
\end{align}
For each fixed $T\in(0,T_{\max})$, the constant $C_2(T)$ is finite and independent of $\hbar$, so the left-hand side of~\eqref{transferfixedtime} tends to $0$ as $\hbar\to0$.

\emph{Step 3: regularity transfer and global existence.}
Fix $t\in[0,T_{\max})$. By Step~2, applied on $[0,T]$ for any $T\in(t,T_{\max})$, $W^\hbar[P^\hbar(t)]\to g(t)$ in $H^\sigma_2$ as $\hbar\to0$. Moreover, \eqref{qhung2} and Definition~\ref{quantumnormsdefinition} give
\begin{align*}
    \sup_{\hbar\in(0,\delta_0]}\norm{W^\hbar[P^\hbar(t)]}_{H^{\sigma'}_2}\leq C_b\jb{t}^{5/2}.
\end{align*}
Along any sequence $\hbar_n\to0$, a subsequence of $(W^{\hbar_n}[P^{\hbar_n}(t)])_n$ converges weakly in $H^{\sigma'}_2$, and the weak limit coincides with the strong $H^\sigma_2$ limit $g(t)$; weak lower semicontinuity of the norm then gives
\begin{align*}
\norm{g(t)}_{H^{\sigma'}_2}
\leq C_b\jb{t}^{5/2},
\qquad t\in[0,T_{\max}).
\end{align*}
If $T_{\max}<\infty$, then $\sup_{t\in[0,T_{\max})}\norm{g(t)}_{H^{\sigma'}_2}\leq C_b\jb{T_{\max}}^{5/2}<\infty$, contradicting the lower-order continuation criterion of Proposition~\ref{prop:vlasovLWP} with $r:=\sigma'\in(5/2,\sigma_0)$. Hence $T_{\max}=\infty$, proving~\eqref{classicalregularitytransfer}.

\emph{Step 4: scattering.}
Fix $0\leq s\leq t<\infty$. Inserting $Q^\hbar_\infty$ and using~\eqref{qhung},
\begin{align*}
\norm{P^\hbar(t)-P^\hbar(s)}_{\mathcal H^\sigma_2}
\leq\norm{P^\hbar(t)-Q^\hbar_\infty}_{\mathcal H^\sigma_2}+\norm{P^\hbar(s)-Q^\hbar_\infty}_{\mathcal H^\sigma_2}
\leq\frac{\alpha_0}{\jb{t}^{3/2}}+\frac{\alpha_0}{\jb{s}^{3/2}},
\end{align*}
uniformly in $\hbar\in(0,\delta_0]$. By the linearity of the Wigner transform and Definition~\ref{quantumnormsdefinition},
\begin{align*}
\norm{P^\hbar(t)-P^\hbar(s)}_{\mathcal H^\sigma_2}
=\norm{W^\hbar[P^\hbar(t)]-W^\hbar[P^\hbar(s)]}_{H^\sigma_2}
\rightarrow\norm{g(t)-g(s)}_{H^\sigma_2}
\end{align*}
as $\hbar\to0$, by the strong convergence of Step~2, whence
\begin{align*}
\norm{g(t)-g(s)}_{H^\sigma_2}
\leq\frac{\alpha_0}{\jb{t}^{3/2}}+\frac{\alpha_0}{\jb{s}^{3/2}}.
\end{align*}
Hence $(g(t))_{t\geq0}$ is Cauchy in $H^\sigma_2$ as $t\to\infty$, and by completeness admits a limit $h_\infty\in H^\sigma_2$. Letting $t\to\infty$ at fixed $s$ in the last display yields $\norm{g(s)-h_\infty}_{H^\sigma_2}\leq\alpha_0\jb{s}^{-3/2}$, which concludes~\eqref{classicalscatteringtransfer}.
\end{proof}

\section{Uniform-in-time semiclassical convergence}\label{uniformtimesection}
Theorem~\ref{theoremonapplication} is now deduced from the estimates of Sections~\ref{Gronwallsection} and~\ref{scatteringproofsection}; the first step transfers the smallness of the quantum data to the classical datum by weak lower semicontinuity, so that scattering and regularity bounds are available on both sides. These bounds feed back into Theorem~\ref{maintheorem1}: up to polynomial factors, the semiclassical error on $[0,T]$ is of size $\hbar^qe^{AT^{9/2}}$, where $q:=\min\set{2,p}$ combines the alignment order with the semiclassical rate, while after time $T$ both solutions are within $O(T^{-3/2})$ of their scattering profiles. The horizon $T=T(\hbar)$ is chosen to balance the two contributions, yielding the logarithmic rate of~\eqref{uniformrate} uniformly for all $t\in[0,\infty]$.

\begin{proof}[Proof of Theorem~\ref{theoremonapplication}]
Let $\overline\sigma>\sigma_0>\sigma'>\sigma$ be the exponents supplied by Theorem~\ref{Smithresult}. By Lemma~\ref{Penroseequivalence}~\ref{Penroseequivalencei}, the uniform Penrose condition~\eqref{Penrose2} with constant $\kappa$ implies the classical Penrose condition~\eqref{Penrose1} with the same constant. Let $\ep_{\mathrm q}$ and $\ep_{\mathrm c}$ denote the smallness thresholds provided by Theorem~\ref{Smithresult} and Proposition~\ref{Vlasovscattering}, respectively, and set $\ep_0:=\min\set{\ep_{\mathrm q},\ep_{\mathrm c}}$. Hypothesis~\eqref{initialconds} therefore allows Theorem~\ref{Smithresult} to be applied directly to the quantum data, while Step~1 shows that $\hin$ satisfies the smallness hypothesis of Proposition~\ref{Vlasovscattering}.

\emph{Step 1: global existence and scattering on both sides.}
On the quantum side, hypothesis~\eqref{initialconds} is precisely the smallness condition~\eqref{quantumsmallness}, so Theorem~\ref{Smithresult} applies directly to the family $(\Qin^\hbar)_{\hbar\in(0,\delta]}$: the Hartree solutions exist globally, the scattering profiles $Q^\hbar_\infty:=\lim_{t\to\infty}P^\hbar(t)$ exist in $\mathcal H^\sigma_2$, and the estimates~\eqref{qhung} and~\eqref{qhung2} hold.

The classical datum inherits the smallness through the alignment. By Lemma~\ref{operatorsWigner}~\ref{intertwiningproperty} and Definition~\ref{quantumnormsdefinition}, hypothesis~\eqref{initialconds} states that
\begin{align*}
\sup_{\hbar\in(0,\delta]}\sum_{\aal\le2}\norm{x^\alpha W^{\hbar}[\Qin^\hbar]}_{H^{\sigma_0}_2}\le\ep_0.
\end{align*}
By the alignment hypothesis~\ref{alignmentbullet}, $W^{\hbar}[\Qin^\hbar]\to\hin$ in $H^\sigma_2$ as $\hbar\to0$; in particular, $\hin$ is real-valued as the strong limit of the real-valued functions $W^{\hbar}[\Qin^\hbar]$, and $x^\alpha W^{\hbar}[\Qin^\hbar]\to x^\alpha\hin$ in the sense of distributions for each $\aal\le2$. Along a sequence $\hbar_n\to0$, each bounded family $(x^\alpha W^{\hbar_n}[Q^{\hbar_n}_{\mathrm{in}}])_n$ admits a subsequence converging weakly in $H^{\sigma_0}_2$, necessarily to $x^\alpha\hin$; passing to a common subsequence and using weak lower semicontinuity and the superadditivity of the limit inferior,
\begin{align*}
\sum_{\aal\le2}\norm{x^\alpha\hin}_{H^{\sigma_0}_2}
\leq\liminf_{n\to\infty}\sum_{\aal\le2}\norm{x^\alpha W^{\hbar_n}[Q^{\hbar_n}_{\mathrm{in}}]}_{H^{\sigma_0}_2}
\leq\ep_0.
\end{align*}
This is the smallness condition~\eqref{classicalsmallness}, so Proposition~\ref{Vlasovscattering} applies to $\hin$: the Vlasov solution $g$ exists globally, the scattering profile $h_\infty:=\lim_{t\to\infty}g(t)$ exists in $H^\sigma_2$, and the estimates~\eqref{classicalscatteringtransfer} and~\eqref{classicalregularitytransfer} hold. Combining the classical and quantum estimates, and keeping the names $\alpha_0$ and $C_b$ for the enlarged constants, for all $t\geq0$,
\begin{align}\label{hung}
\norm{g(t)-h_\infty}_{H^\sigma_2}
+
\sup_{\hbar\in(0,\delta]}
\norm{P^\hbar(t)-Q^\hbar_\infty}_{\mathcal H^\sigma_2}
&\leq
\frac{\alpha_0}{\jb{t}^{3/2}},\\
\norm{g(t)}_{H^{\sigma'}_2}
+
\sup_{\hbar\in(0,\delta]}
\norm{P^\hbar(t)}_{\mathcal H^{\sigma'}_2}
&\leq
C_b\jb{t}^{5/2}.
\label{hung2}
\end{align}

\emph{Step 2: convergence of the scattering profiles.}
Fix $T\geq1$ and set $q:=\min\set{2,p}\in(0,2]$. By~\eqref{hung2} and $\sigma'\geq\sigma+3$, hypothesis~\ref{amaintheorem1} of Theorem~\ref{maintheorem1} holds on $[0,T]$ with $M_T:=C_b\jb{T}^{5/2}$. Since $\hbar\leq\delta\leq1$, the alignment hypothesis~\ref{alignmentbullet} yields, for all $t\in[0,T]$ and $\hbar\in(0,\delta]$,
\begin{align*}
\norm{g(t)-W^{\hbar}[P^\hbar(t)]}_{H^\sigma_2}
\leq
\br{C_0C_1(T)+C_2(T)}\hbar^q,
\end{align*}
with $C_1(T)$ and $C_2(T)$ as in~\eqref{Cconstants}. Inserting $M_T=C_b\jb{T}^{5/2}$ into~\eqref{Cconstants} and using $\jb{T}\leq2T$ for $T\geq1$, there exist constants $A,\beta>0$, depending only on $\sigma$, $K$, $\norm{\mu}_{H^{\sigma+3}_2}$, $C_b$ and $C_0$, such that
\begin{align}\label{huge2}
\norm{g(t)-W^{\hbar}[P^\hbar(t)]}_{H^\sigma_2}\leq\beta\hbar^qe^{AT^{9/2}}T^7
\end{align}
for all $t\in[0,T]$ and $\hbar\in(0,\delta]$. Decomposing the difference between the scattering profiles at time $T$, and using the linearity and isometry of the Wigner transform (Definition~\ref{quantumnormsdefinition}) together with~\eqref{hung} for the outer terms and~\eqref{huge2} for the middle term,
\begin{align}\label{badger}
\norm{W^{\hbar}[Q^\hbar_\infty]-h_\infty}_{H^\sigma_2}
&\leq
\norm{Q^\hbar_\infty-P^\hbar(T)}_{\mathcal H^\sigma_2}
+
\norm{W^{\hbar}[P^\hbar(T)]-g(T)}_{H^\sigma_2}
+
\norm{g(T)-h_\infty}_{H^\sigma_2}\nonumber\\
&\leq
\alpha_0 T^{-3/2}
+
\beta\hbar^qe^{AT^{9/2}}T^7.
\end{align}
The horizon $T$ is now chosen as a function of $\hbar$ to balance the two terms.
Choose $L_*\geq\max\set{1,2A/q}$ sufficiently large that $2\ln L\leq\tfrac{q}{2}L$ for every $L\geq L_*$, and set $\delta_*:=\min\set{\delta,e^{-L_*}}$. For $\hbar\in(0,\delta_*]$, write $L_\hbar:=\ln(\hbar^{-1})$, so that $L_\hbar\geq L_*$, and define
\begin{align*}
T(\hbar):=\sbr{\frac{1}{A}\br{qL_\hbar-2\ln L_\hbar}}^{2/9}.
\end{align*}
The choice of $L_*$ gives
\begin{align*}
\frac{q}{2}L_\hbar
\leq qL_\hbar-2\ln L_\hbar
\leq qL_\hbar,
\end{align*}
whence $T(\hbar)\geq\sbr{\tfrac{q}{2A}L_*}^{2/9}\geq1$, so~\eqref{badger} applies with $T=T(\hbar)$. Moreover, $T(\hbar)\sim L_\hbar^{2/9}$, with constants depending on $p$ and $A$, and
\begin{align*}
\hbar^qe^{AT(\hbar)^{9/2}}
=e^{-qL_\hbar}e^{qL_\hbar-2\ln L_\hbar}
=L_\hbar^{-2}.
\end{align*}
Consequently,
\begin{align}\label{balancedbounds}
T(\hbar)^{-3/2}\lesssim L_\hbar^{-1/3},
\qquad
\hbar^qe^{AT(\hbar)^{9/2}}T(\hbar)^7
\lesssim L_\hbar^{-2}L_\hbar^{14/9}=L_\hbar^{-4/9}\leq L_\hbar^{-1/3},
\end{align}
using $L_\hbar\geq1$ in the final inequality. Inserting~\eqref{balancedbounds} into~\eqref{badger} yields~\eqref{uniformrate} at $t=\infty$ for all $\hbar\in(0,\delta_*]$.

\emph{Step 3: uniform-in-time convergence.}
Fix $\hbar\in(0,\delta_*]$ and $t\in[0,\infty)$. If $t\leq T(\hbar)$, then~\eqref{huge2} with $T=T(\hbar)$ and~\eqref{balancedbounds} give
\begin{align*}
\norm{W^{\hbar}[P^\hbar(t)]-g(t)}_{H^\sigma_2}
\leq\beta\hbar^qe^{AT(\hbar)^{9/2}}T(\hbar)^7
\lesssim L_\hbar^{-1/3}.
\end{align*}
If $t\geq T(\hbar)$, inserting the scattering profiles and using~\eqref{hung} for the outer terms, together with the linearity and isometry of the Wigner transform, and~\eqref{uniformrate} at $t=\infty$ for the middle term,
\begin{align*}
\norm{W^{\hbar}[P^\hbar(t)]-g(t)}_{H^\sigma_2}
&\leq
\norm{P^\hbar(t)-Q^\hbar_\infty}_{\mathcal H^\sigma_2}
+
\norm{W^{\hbar}[Q^\hbar_\infty]-h_\infty}_{H^\sigma_2}
+
\norm{h_\infty-g(t)}_{H^\sigma_2}\\
&\leq
\alpha_0\jb{t}^{-3/2}
+
C_\infty L_\hbar^{-1/3},
\end{align*}
and $t\geq T(\hbar)$ implies $\jb{t}^{-3/2}\leq T(\hbar)^{-3/2}\lesssim L_\hbar^{-1/3}$. Combining the two cases concludes the bound~\eqref{uniformrate} for all $t\in[0,\infty]$, after enlarging $C_\infty$ if necessary.
\end{proof}

\appendix
\section{Well-posedness and propagation of regularity}\label{PropagationAppendixReference}
This appendix establishes local well-posedness for the Vlasov equation in $H^\sigma_2$, with a continuation criterion in lower-order norms. It then records a conditional propagation statement for the Hartree equation in $\mathcal H^\sigma_2$, uniform in $\hbar$. The results are formulated for the framed solutions $g(t)$ and $P^\hbar(t)$ obtained by factoring out the free transport and Schr\"odinger evolutions.

\subsection{Local well-posedness and continuation for the Vlasov equation}
The following proposition supplies the local well-posedness and lower-order continuation criterion for the moving-frame Vlasov equation~\eqref{equationforg(t)} used in the proof of Proposition~\ref{Vlasovscattering}.

\begin{proposition}[Local well-posedness and continuation for the Vlasov equation]\label{prop:vlasovLWP}
Let $\sigma>5/2$. Assume $\mu\in H^{\sigma+1}_2(\R^3)$ and $[K]_1<\infty$. Then, for every real-valued $\hin\in H^\sigma_2(\R^3\times\R^3)$, there exist $T_{\max}\in(0,\infty]$ and a unique maximal real-valued solution
\begin{align*}
g\in C([0,T_{\max});H^\sigma_2(\R^3\times\R^3))
\end{align*}
of the moving-frame Vlasov equation~\eqref{equationforg(t)} with $g(0)=\hin$. Moreover, if $T_{\max}<\infty$, then, for $r\in(5/2,\sigma]$,
\begin{align*}
\limsup_{t\to T_{\max}^-}\norm{g(t)}_{H^r_2}=\infty.
\end{align*}
\end{proposition}

\begin{proof}
\emph{Step 1: a priori estimate.} For every $r\in(5/2,\sigma]$, solutions obey
\begin{align}\label{lwpapriori}
\frac{d}{dt}\norm{g(t)}^2_{H^\sigma_2}\lesssim_{\sigma,r,K}\br{\norm{\mu}_{H^{\sigma+1}_2}+\jb{t}\norm{g(t)}_{H^r_2}}\norm{g(t)}^2_{H^\sigma_2},
\end{align}
which follows from Lemma~\ref{lemmaboundsonLN} for the linear term and Lemma~\ref{symmetrylemma}, applied with $s:=r-1>3/2$, for the nonlinear term: the direct bound on $N$ loses one derivative in its second argument, and the symmetric cancellation closes the energy estimate at the level $H^\sigma_2$.

\emph{Step 2: stability and uniqueness.} Fix $s:=\sigma-1>3/2$, and let $w:=g_1-g_2$ denote the difference of two real-valued solutions, which satisfies
\begin{align*}
\p_tw=L_\mu[w]+N[w,g_1]+N[g_2,w].
\end{align*}
The three terms are estimated separately. The linear term satisfies $\norm{L_\mu[w]}_{H^0_2}\lesssim_K\norm{\mu}_{H^1_2}\norm{w}_{H^0_2}$, by Lemma~\ref{lemmaboundsonLN} with $\sigma=0$. 

For $N[w,g_1]$, recall the representation~\eqref{mathcalN}. The weights enter by direct differentiation: for $\aal\leq2$, the Leibniz rule splits $D^\alpha_\eta\mathcal N[w,g_1]$ into the term in which $D^\alpha_\eta$ falls on $\widehat{g_1}$ and lower-order terms in which an $\eta$-derivative removes the factor $\eta-kt$. Using $\abs{\eta-kt}\lesssim\jb{t}\jb{k-\ell,\eta-t\ell}$, every term is of the form treated by Lemma~\ref{importantlemma}~\ref{importantlemmag_1H^0_2g_2H^s}, and, since $s+1=\sigma$, this gives $\norm{N[w,g_1]}_{H^0_2}\lesssim_{\sigma,K}\jb{t}\norm{w}_{H^0_2}\norm{g_1}_{H^\sigma_2}$. 

The remaining term $N[g_2,w]=\nabla V_{g_2}(t,z+t\xi)\cdot(\nabla_\xi-t\nabla_z)w$, with $V_{g_2}$ the self-consistent potential~\eqref{selfconsistentpotential}, is the transport part. For real-valued $v\in H^1$, writing $v(\nabla_\xi-t\nabla_z)v=\tfrac12(\nabla_\xi-t\nabla_z)(v^2)$, grouping the coefficients of $\nabla_z$ and $\nabla_\xi$, and integrating by parts,
\begin{align*}
\Re\jb{v,N[g_2,v]}_{L^2}
=\frac12\int_{\R^6}\br{-t\nabla V_{g_2},\nabla V_{g_2}}\cdot\nabla_{z,\xi}(v^2)\,dzd\xi
=-\frac12\int_{\R^6}\nabla_{z,\xi}\cdot\br{-t\nabla V_{g_2},\nabla V_{g_2}}v^2\,dzd\xi
=0,
\end{align*}
since $\nabla_{z,\xi}\cdot\br{-t\nabla V_{g_2},\nabla V_{g_2}}=-t\Delta V_{g_2}+t\Delta V_{g_2}=0$. For each $\aal\leq2$, Lemma~\ref{decompositionlemma} decomposes $\xi^\alpha N[g_2,w]=N[g_2,\xi^\alpha w]+N_{<\alpha}[g_2,w]$; the cancellation, applied with $v=\xi^\alpha w$, removes the first term from the inner product, and the remainders $N_{<\alpha}[g_2,w]$ are bounded via~\eqref{VlasovremainderboundN} and Lemma~\ref{importantlemma}~\ref{importantlemmag_1H^s_2g_2L^2}. Combining the three estimates yields the low-norm bound
\begin{align}\label{lwpstability}
\frac{d}{dt}\norm{w(t)}^2_{H^0_2}\lesssim_{\sigma,K}\br{\norm{\mu}_{H^{1}_2}+\jb{t}\br{\norm{g_1(t)}_{H^\sigma_2}+\norm{g_2(t)}_{H^\sigma_2}}}\norm{w(t)}^2_{H^0_2},
\end{align}
and Gr\"onwall's inequality applied to~\eqref{lwpstability} gives uniqueness.

\emph{Step 3: construction and continuation.} With~\eqref{lwpapriori}, taken with $r=\sigma$, and~\eqref{lwpstability} in hand, a local solution is constructed by the standard approximation scheme, solving the equation with the frequencies smoothly truncated and passing to the limit; the argument is classical and uses nothing specific to the weighted spaces; see, for instance,~\cite[Chapter~3]{MajdaBertozzi2002}. The local existence time is bounded below in terms of the current $H^\sigma_2$ norm, the fixed parameters and the initial time, so the standard restart argument produces a unique maximal solution on $[0,T_{\max})$ together with the continuation criterion in the top norm: if $T_{\max}<\infty$, then $\limsup_{t\to T_{\max}^-}\norm{g(t)}_{H^\sigma_2}=\infty$.

\emph{Step 4: the lower-order criterion.} Applying Gr\"onwall's inequality to~\eqref{lwpapriori} gives, for every $r\in(5/2,\sigma]$ and all $t\in[0,T_{\max})$,
\begin{align*}
\norm{g(t)}_{H^\sigma_2}\leq\norm{\hin}_{H^\sigma_2}\exp\sbr{C_{\sigma,r,K}\br{t\norm{\mu}_{H^{\sigma+1}_2}+\int_0^t\jb{s}\norm{g(s)}_{H^r_2}\,ds}}.
\end{align*}
Suppose now that $T_{\max}<\infty$ and that, for some $r\in(5/2,\sigma]$,
\begin{align*}
\sup_{t\in[0,T_{\max})}\norm{g(t)}_{H^r_2}<\infty.
\end{align*}
Then the preceding estimate keeps the $H^\sigma_2$ norm bounded up to $T_{\max}$, contradicting the top-norm criterion.
\end{proof}

\subsection{Propagation of regularity for the Hartree equation}\label{PropagationAppendixReferenceHARTREE}
For the Hartree equation, Proposition~\ref{Hartreeregularityproposition} propagates high regularity uniformly in $\hbar$: for any family of solutions whose lower-order norm remains bounded, the full $\mathcal H^\sigma_2$ norm is controlled in terms of the initial data alone. 

\begin{proposition}[Hartree regularity propagation]\label{Hartreeregularityproposition}
Let $T>0$, $\sigma>5/2$ and $\delta\in(0,1]$. Assume $\mu\in H^{\sigma+1}_2(\R^3)$ and $[K]_2<\infty$. For every $\hbar\in(0,\delta]$, let $Q^\hbar\in C([0,T];\mathcal H^\sigma_2)$ solve the Hartree equation~\eqref{NH} with self-adjoint initial datum $\Qin^\hbar\in\mathcal H^\sigma_2$, and set $P^\hbar(t):=e^{-i\frac{\hbar}{2}t\Delta}Q^\hbar(t)e^{i\frac{\hbar}{2}t\Delta}$. If, for some $r\in(5/2,\sigma]$,
\begin{align*}
\sup_{\hbar\in(0,\delta]}\norm{\Qin^\hbar}_{\mathcal H^\sigma_2}
+
\sup_{\hbar\in(0,\delta]}\sup_{t\in[0,T]}\norm{P^\hbar(t)}_{\mathcal H^r_2}
<\infty,
\end{align*}
then
\begin{align*}
\sup_{\hbar\in(0,\delta]}\sup_{t\in[0,T]}\norm{P^\hbar(t)}_{\mathcal H^\sigma_2}
<\infty.
\end{align*}
\end{proposition}
\begin{remark}
Only the conditional statement is needed here, since the Hartree solutions are supplied by Theorem~\ref{Smithresult}. A full quantum analogue of Proposition~\ref{prop:vlasovLWP} nevertheless holds: for $[K]_2<\infty$ and $\hbar\in(0,1]$, the argument extends verbatim to the framed Wigner equation~\eqref{eqforg^hbar(t)}, with the transport cancellation of Step~2 replaced by the sign change of the sine multiplier in the proof of Lemma~\ref{symmetrylemma}, and yields local well-posedness in $\mathcal H^\sigma_2$ with existence time and continuation criterion uniform in $\hbar$.
\end{remark}
\begin{proof}
Fix $\hbar\in(0,\delta]$ and set $g^\hbar(t):=W^{\hbar}[P^\hbar(t)]$. The Hartree evolution and conjugation by the free Schr\"odinger group preserve self-adjointness, so $g^\hbar$ is real-valued. Moreover, $g^\hbar$ solves~\eqref{eqforg^hbar(t)}, and each $H_2$ norm of $g^\hbar$ equals the corresponding $\mathcal H_2$ norm of $P^\hbar$, by Definition~\ref{quantumnormsdefinition}.

The quantum estimates in Lemmas~\ref{lemmaboundsonLN} and~\ref{symmetrylemma}, applied with $s:=r-1>3/2$, give
\begin{align*}
\frac{d}{dt}\norm{g^\hbar(t)}_{H^\sigma_2}^2
\lesssim_{\sigma,r,K}
\br{\norm{\mu}_{H^{\sigma+1}_2}+\jb{t}\norm{g^\hbar(t)}_{H^r_2}}\norm{g^\hbar(t)}_{H^\sigma_2}^2.
\end{align*}
This is the quantum analogue of~\eqref{lwpapriori}. Gr\"onwall's inequality therefore yields
\begin{align*}
\norm{P^\hbar(t)}_{\mathcal H^\sigma_2}
\leq
\norm{\Qin^\hbar}_{\mathcal H^\sigma_2}
\exp\sbr{C_{\sigma,r,K}\br{t\norm{\mu}_{H^{\sigma+1}_2}+\int_0^t\jb{s}\norm{P^\hbar(s)}_{\mathcal H^r_2}\,ds}},
\end{align*}
where $C_{\sigma,r,K}$ is independent of $\hbar$ and $T$. Taking the supremum over $t\in[0,T]$ and $\hbar\in(0,\delta]$ proves the result.
\end{proof}

\section{Japanese bracket inequalities}\label{appendixB}
Standard inequalities for the Japanese bracket $\jb{x}:=(1+\abs{x}^2)^{1/2}$ and related weights are recorded. These extend directly to the double bracket $\jb{x,y}:=\jb{(x,y)}=(1+\abs{x}^2+\abs{y}^2)^{1/2}$ for $(x,y)\in \R^{6}$.
\begin{lemma}\label{japanesebracketlemma}
    Let $x,y\in\R^d$.
    \begin{enumerate}[label=(\roman*)]
        \item Triangle inequality: for $\sigma\geq0$, there exists $C_\sigma>0$ such that\label{JBtriangle}
        \begin{align*}
            \jb{x+y}^\sigma\leq C_\sigma\br{\jb{x}^\sigma+\jb{y}^\sigma}.
        \end{align*}
        \item Peetre inequality:\label{JBproduct}
        \begin{align*}
            \jb{x+y}\leq\sqrt{2}\jb{x}\jb{y}.
        \end{align*}
        \item Shear bound: for $t\in\R$,\label{JBshear}
        \begin{align*}
            \jb{x+ty}\leq\sqrt{2}\jb{t}\jb{x,y}.
        \end{align*}
        \item Symbol bound: for $\sigma\geq0$ and any multi-index $\beta$,\label{JBsymbol}
        \begin{align*}
            \abs{D^\beta_x\jb{x}^\sigma}\lesssim_{\sigma,\beta}\jb{x}^{\sigma-\abs{\beta}}.
        \end{align*}
        \item Reverse triangle inequality: for $\sigma\geq1$,\label{JBreverse}
        \begin{align*}
            \abs{\jb{x}^\sigma-\jb{y}^\sigma}\leq \sigma\br{\jb{x}^{\sigma-1}+\jb{y}^{\sigma-1}}\abs{x-y}.
        \end{align*}
    \end{enumerate}
\end{lemma}
\begin{proof}
    Throughout, write $\jb{x}=\abs{(1,x)}$ with $\abs{\cdot}$ the Euclidean norm on $\R^{d+1}$.

    \ref{JBtriangle} The triangle inequality in $\R^{d+1}$ gives $\jb{x+y}=\abs{(1,x)+(0,y)}\leq\jb{x}+\abs{y}\leq\jb{x}+\jb{y}$. Raising to the power $\sigma$, the subadditivity bound $(a+b)^\sigma\leq a^\sigma+b^\sigma$ for $a,b\geq0$ concludes for $0\leq\sigma\leq1$, and the convexity bound $(a+b)^\sigma\leq2^{\sigma-1}\br{a^\sigma+b^\sigma}$ concludes for $\sigma\geq1$; in particular $C_\sigma=\max\set{1,2^{\sigma-1}}$ suffices.

    \ref{JBproduct} Expanding the right-hand side,
    \begin{align*}
        1+\abs{x+y}^2\leq1+2\abs{x}^2+2\abs{y}^2\leq2\br{1+\abs{x}^2}\br{1+\abs{y}^2}.
    \end{align*}

    \ref{JBshear} By the triangle inequality, $\abs{t}\leq\jb{t}$, and the Cauchy--Schwarz inequality,
    \begin{align*}
        \abs{x+ty}\leq\abs{x}+\abs{t}\abs{y}\leq\jb{t}\br{\abs{x}+\abs{y}}\leq\sqrt2\jb{t}\abs{(x,y)},
    \end{align*}
    and hence, using $\jb{t}\geq1$,
    \begin{align*}
        1+\abs{x+ty}^2\leq1+2\jb{t}^2\abs{(x,y)}^2\leq2\jb{t}^2\jb{x,y}^2.
    \end{align*}

    \ref{JBsymbol} The case $\beta=0$ is immediate, so suppose $\aab\geq1$. Since $\jb{x}^\sigma=(1+\abs{x}^2)^{\sigma/2}$, the repeated chain rule gives $D_x^\beta\jb{x}^\sigma$ as a finite linear combination, with coefficients depending only on $\sigma$ and $\beta$, of terms of the form
    \begin{align*}
        (1+\abs{x}^2)^{\sigma/2-r}\prod_{\ell=1}^r D_x^{\gamma_\ell}(1+\abs{x}^2),
    \end{align*}
    where $1\leq r\leq\aab$, the multi-indices $\gamma_\ell$ are non-zero, and $\sum_{\ell=1}^r\abs{\gamma_\ell}=\aab$. Terms with some $\abs{\gamma_\ell}>2$ vanish. For the remaining terms, $\abs{D_x^{\gamma_\ell}(1+\abs{x}^2)}\lesssim\jb{x}^{2-\abs{\gamma_\ell}}$, and therefore $\prod_{\ell=1}^r\abs{D_x^{\gamma_\ell}(1+\abs{x}^2)}\lesssim\jb{x}^{2r-\aab}$. Since $(1+\abs{x}^2)^{\sigma/2-r}=\jb{x}^{\sigma-2r}$, each term is bounded by $\jb{x}^{\sigma-\aab}$, and summing over the finitely many terms concludes.

    \ref{JBreverse} Since $\jb{x}=\abs{(1,x)}$ is the restriction of a norm to an affine subspace, it is convex, so $\jb{y+a(x-y)}\leq(1-a)\jb{y}+a\jb{x}\leq\max\set{\jb{x},\jb{y}}$ for $a\in[0,1]$. By the gradient identity in the proof of~\ref{JBsymbol}, $\abs{\nabla\jb{z}^\sigma}=\sigma\abs{z}\jb{z}^{\sigma-2}\leq\sigma\jb{z}^{\sigma-1}$, and the mean value theorem along the segment from $y$ to $x$ gives
    \begin{align*}
        \abs{\jb{x}^\sigma-\jb{y}^\sigma}\leq\sigma\max\set{\jb{x},\jb{y}}^{\sigma-1}\abs{x-y}\leq\sigma\br{\jb{x}^{\sigma-1}+\jb{y}^{\sigma-1}}\abs{x-y},
    \end{align*}
    using $\sigma\geq1$ in the final step.
\end{proof}
\section*{Acknowledgements}
The author is grateful for valuable discussions with Clément Mouhot and Chiara Saffirio during a residency at the Simons Laufer Mathematical Sciences Institute in Berkeley, California, in the Fall 2025 semester. The residency was supported by the National Science Foundation under Grant No.\ DMS-2424139.

\printbibliography

\end{document}